\documentclass[11pt]{amsart}
\usepackage{amsmath,amssymb,amsthm}
\usepackage{tikz,xstring,ifthen}
\usepackage[numbers,comma,square,sort&compress]{natbib}
\usepackage{mathabx}   
\usepackage{stackengine}
\usepackage{caption}
\usepackage{graphics}
\usepackage{subcaption} 
\RequirePackage{color}
\RequirePackage[colorlinks,urlcolor=my-blue,linkcolor=my-red,citecolor=my-green]{hyperref}

\numberwithin{figure}{section}

\usepackage{enumerate}

\usepackage{scalerel}    
\usepackage{stmaryrd}   
\usepackage{mathabx}   

\definecolor{my-blue}{rgb}{0.0,0.0,0.6}
\definecolor{my-red}{rgb}{0.5,0.0,0.0}
\definecolor{my-green}{rgb}{0.0,0.5,0.0}
\definecolor{nicos-red}{rgb}{0.75,0.0,0.0}

\newtheorem{theorem}{\sc Theorem}[section]
\newtheorem{lemma}[theorem]{\sc Lemma}

\newtheorem{corollary}[theorem]{\sc Corollary}

\numberwithin{equation}{section}
\theoremstyle{remark}
\newtheorem{remark}[theorem]{Remark}

\newcommand{\be}{\begin{equation}}
\newcommand{\ee}{\end{equation}}
\newcommand{\beq}{\begin{equation}}
\newcommand{\eeq}{\end{equation}}

\newcommand{\nn}{\nonumber}
\providecommand{\abs}[1]{\vert#1\vert}

\newcommand{\fl}[1]{\lfloor{#1}\rfloor} 
\newcommand{\ce}[1]{\lceil{#1}\rceil}

\def\cC{\mathcal{C}}  
\def\cF{\mathcal{F}}

\def\cI{\mathcal{I}}

\def\bE{\mathbb{E}} 

\def\bP{\mathbb{P}}

\def\bR{\mathbb{R}}
\def\bZ{\mathbb{Z}}

 \def\Z{\bZ}  \def\R{\bR}

\def\evec{e}

\def\zvec{\mathbf{z}}

\def\Jvec{\mathbf{J}}
\def\Yvec{\mathbf{Y}}

\def\w{\omega}

\def\e{\varepsilon}

\def\ddd{\displaystyle}

\def\m1{\mathbf{1}}

 \def\Vvv{{\rm\mathbb{V}ar}}  \def\Cvv{{\rm\mathbb{C}ov}} 
 \def\wt{\widetilde}  \def\wh{\widehat}  \def\wc{\widecheck}
 
\def\E{\bE}
\def\P{\bP} 

\def\funct lp{L} 
\def\funct lpbar{\bar L} 

\def\Gpp{G}
\def\gpp{g}

\newcommand{\argmax}[1]{\underset{#1}{\arg\max}\;}

\def\geoup#1{#1^{\scaleobj{1.2}\star}}        
\def\geodown#1{#1_{\scaleobj{1.2}\star}}

\def\xiup#1#2{#1^{\scaleobj{1.2}\star}_{#2}}       
\def\xidown#1#2{#1_{{\scaleobj{1.2}\star} #2}}

\newcommand{\rhodown}[1]{{#1}_{\scaleobj{1.2}\star}} 
\newcommand{\rhoup}[1]{{#1}^{\scaleobj{1.2}\star}} 

\definecolor{darkgreen}{rgb}{0.0,0.5,0.0}
\definecolor{darkblue}{rgb}{0.0,0.0,0.3}
\definecolor{nicosred}{rgb}{0.65,0.1,0.1}
\definecolor{light-gray}{gray}{0.7}
\allowdisplaybreaks[1]

\usepackage{filemod}

\def\cif1{v}   

\def\deq{\overset{d}=}

\def\eventcB{U}
 
\def\eventcD{D}

\def\Dop{D}    
\def\Rop{R}  
\def\Sop{S}    

\def\arr{a} 
\def\barr{b}
\def\serv{s}
\def\depa{d} 
\def\sojo{t}
\def\arrv{\mathbf\arr}  \def\barrv{\mathbf\barr} 
\def\servv{\mathbf\serv}
\def\depav{\mathbf\depa} 
\def\sojov{\mathbf\sojo}

\def\aaa#1{\alpha[#1]}  

\def\rim#1{\widehat{#1}} 

\newcommand{\lzb}{\llbracket}   
\newcommand{\rzb}{\rrbracket}   

\newcommand\bbullet{{\raisebox{0.5pt}{\scaleobj{0.6}{\bullet}}}} 
\newcommand\brbullet{{\raisebox{-1pt}{\scaleobj{0.5}{\bullet}}}} 

\usepackage[nomessages]{fp}
\newcounter{usedm}
\newcounter{usedn}
\newcommand*{\lppwl}[6]{
	
	\FPeval{\m}{round(#5*(#3-#1)/(#3+#4-#1-#2),0)}
	
	\FPeval{\n}{round(#5*(#4-#2)/(#3+#4-#1-#2),0)}
	
	\FPeval{\stlength}{(#3+#4-#1-#2)/#5}
	
	\setcounter{usedm}{0}
	\setcounter{usedn}{0}
	
	\foreach\i in{1,...,{#5}}{
		\FPrandom{\x}
		
		\FPeval{\y}{(\m-\theusedm)/(\m+\n-\theusedm-\theusedn)}
		
		\FPeval{\startx}{#1+(\theusedm)*\stlength}
		\FPeval{\starty}{#2+(\theusedn)*\stlength}
		
		\ifthenelse{\lengthtest{\x pt > \y pt}}{
			\FPset{\endx}{\startx}
			\FPadd{\endy}{\starty}{\stlength}
			\stepcounter{usedn}
		}
		{
			\FPadd{\endx}{\startx}{\stlength}
			\FPset{\endy}{\starty}
			\stepcounter{usedm}
		}
		\draw[#6](\startx,\starty) -- (\endx,\endy);
	};
}

\begin{document}

 
\title[Bi-infinite geodesics]{Non-existence of bi-infinite geodesics in the exponential corner growth model}

\author[M.~Bal\'azs]{M\'arton Bal\'azs}
\address{M\'arton Bal\'azs\\ University of Bristol\\  School of Mathematics\\ Fry Building\\ Woodland Rd.\\   Bristol BS8 1UG\\ UK.}
\email{m.balazs@bristol.ac.uk}
\urladdr{https://people.maths.bris.ac.uk/~mb13434/}
\thanks{M.\ Bal\'azs was partially supported by EPSRC's EP/R021449/1 Standard Grant.}

\author[O.~Busani]{Ofer Busani}
\address{Ofer Busani\\ University of Bristol\\  School of Mathematics\\ Fry Building\\ Woodland Rd.\\   Bristol BS8 1UG\\ UK.}
\email{o.busani@bristol.ac.uk}
\urladdr{https://people.maths.bris.ac.uk/~di18476/}
\thanks{O. Busani was supported by EPSRC's EP/R021449/1 Standard Grant.}

\author[T.~Sepp\"al\"ainen]{Timo Sepp\"al\"ainen}
\address{Timo Sepp\"al\"ainen\\ University of Wisconsin-Madison\\  Mathematics Department\\ Van Vleck Hall\\ 480 Lincoln Dr.\\   Madison WI 53706-1388\\ USA.}
\email{seppalai@math.wisc.edu}
\urladdr{http://www.math.wisc.edu/~seppalai}
\thanks{T.\ Sepp\"al\"ainen was partially supported by  National Science Foundation grant  DMS-1854619 and by  the Wisconsin Alumni Research Foundation.} 

\keywords{bi-infinite, corner growth model,  directed percolation, geodesic, random growth model, last-passage percolation, queues}
\subjclass[2000]{60K35, 65K37} 
\date{\today} 

\begin{abstract}
This paper gives a self-contained proof of the non-existence of nontrivial bi-infinite geodesics in directed planar last-passage percolation with exponential weights. The techniques used are couplings, coarse graining, and control of geodesics through planarity and estimates derived from increment-stationary versions of the last-passage percolation process. 
 \end{abstract}
\maketitle

\setcounter{tocdepth}{1}
\tableofcontents

\section{Introduction}

\subsection{Bi-infinite geodesics in random growth} 
Since their inception over 50 years ago in the work of Eden \cite{eden-61} and Hammersley and Welsh \cite{Hamm-wels-65}, random growth models have been   central drivers of the mathematical theory 
of  spatial random processes. 
 Particularly important classes of growth models are undirected {\it first-passage percolation} (FPP) and  directed  {\it last-passage percolation} (LPP) where growth proceeds   along optimal paths called {\it geodesics}.    The structure of these geodesics has been a challenging object of study.   

Under natural assumptions, the existence of a geodesic between  two  points in space is   straightforward. A compactness argument   gives the existence of  a   semi-infinite geodesic, that is, a one-sided infinite    path  
that furnishes the geodesic between any two of its points. 
The existence or non-existence of   bi-infinite geodesics  has turned out to be a very hard problem. 
This question   was first posed to H.~Kesten by H.~Furstenberg in the context of FPP  \cite[p.~258]{kest-86}.  Apart from its significance for random growth, this existence issue  is   tied to questions about ground states of certain disordered models of statistical physics (\cite[p.~105]{auff-damr-hans-17-ams}, \cite[Ch.~1]{newm-97-book}).


The development of mathematical techniques for    infinite  geodesics in two-dimensional FPP and LPP began with the work of C.~Newman and coauthors  
in the 1990s \cite{newm-icm-95}. 
    Licea and Newman \cite{lice-newm-96}  ruled out   directed bi-infinite geodesics with given  direction in an  unknown set of full Lebesgue measure. 
 Much more  recently, a bi-infinite geodesic in any fixed direction has been ruled out, but  subject to a local regularity condition on the limit shape,   by \cite{geor-rass-sepp-17-geod} in LPP and by   \cite{ahlb-hoff, damr-hans-17} in FPP.  The new approach here was based on Busemann functions.  
  Bi-infinite FPP geodesics have also been ruled out in certain restricted subsets of the lattice such as half-planes \cite{auff-damr-hans-15,wehr-woo}.   
However, despite all the effort,   a feasible strategy for solving the bi-infinite existence  problem in FPP without restrictive  assumptions  is not presently visible.  


In   exactly solvable planar directed LPP,  techniques have  evolved to the point where  the existence problem can be given a complete solution.   The first proof of the nonexistence of bi-infinite geodesics in planar LPP with exponential weights appeared in the 2018 preprint \cite{basu-hoff-sly-arxiv-18} of Basu, Hoffman and Sly.  Their work relies on   fluctuation and moderate deviation  estimates for the passage times  that come from integrable probability. 
These estimates were originally  obtained   through combinatorial analysis,
asymptotic analysis of Fredholm determinants, and random matrix methods.  Further results  from these estimates were   derived  in the preprint \cite{basu-sido-sly-arxiv-14} by Basu, Sidoravicius and Sly, in particular to control transversal fluctuations of geodesics,  and then applied to the bi-infinite geodesic problem in \cite{basu-hoff-sly-arxiv-18}.     

The elaborate multilayered effort behind \cite{basu-hoff-sly-arxiv-18} is remarkable.  
   It raises an obvious  question, namely,  whether  ruling out bi-infinite geodesics requires  the power of integrable probability. 

The present paper answers this question in the negative by providing a second  proof of the nonexistence of bi-infinite geodesics that reduces the technical requisites considerably.   Nothing beyond standard probability tools such as coupling and coarse graining is needed.  The features specific to the exponential LPP utilized are the  independence properties of its stationary version.  These  independence properties cannot all hold for general i.i.d.\ weights.  But if they were replaced with sufficient mixing,  the estimates behind  our proof would remain provable in weaker form.    

Next we state the  main result and then relate our proof to existing  literature.  In particular, we contrast our work with  \cite{basu-hoff-sly-arxiv-18} in more detail. 


\subsection{Main result} 
   The model studied   is a version of nearest-neighbor directed LPP on the planar integer lattice, also known as the {\it corner growth model} (CGM).  
 	Let $\omega=\{\omega_x\}_{x\in \mathbb{Z}^2}$ be an assignment of random weights on the vertices of $\mathbb{Z}^2$.  The weights $\omega_x$ are independent and identically distributed  (i.i.d) random variables with rate one exponential distribution, that is,  $\P(\w_x>t)=e^{-t}$ for each $x\in\Z^2$ and real $t\ge 0$.    The last-passage value $G_{x,y}$ for coordinatewise ordered points $x\le y$ on $\Z^2$ is defined by 
		\be\label{v:G}  
	\Gpp_{x,y}=\max_{x_{\brbullet}\,\in\,\Pi_{x,y}}\sum_{k=0}^{\abs{y-x}_1}\w_{x_k}, 
	\ee
	where   $\Pi_{x,y}$ is the set of  {\it nearest-neighbor  up-right  paths} $x_{\bbullet}=(x_k)_{k=0}^n$  that start at  $x_0=x$ and  end at $x_n=y$ with $n=\abs{y-x}_1=$ the number of nearest-neighbor steps from $x$ to $y$. Such paths  are defined by the requirement  $x_{k+1}-x_k\in\{e_1,e_2\}$.  (See Figure \ref{fig:cgm}.)   When the weights have a continuous distribution such as the exponential, \eqref{v:G} has  a unique maximizing path $\pi^{x,y}\in\Pi_{x,y}$ called the ({\it point-to-point} or {\it finite})  {\it geodesic}.  
	
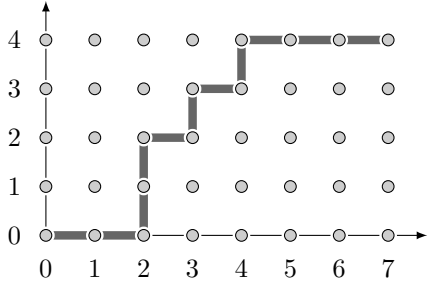
\begin{figure}[t] 
	\begin{center}
		\begin{tikzpicture}[>=latex, scale=0.65]  

			\definecolor{sussexg}{gray}{0.7}
			\definecolor{sussexp}{gray}{0.4} 

			\draw[<->](0,4.8)--(0,0)--(7.8,0);
			
			\draw[line width = 3.2 pt, color=sussexp](0,0)--(2,0)--(2,2)--(3,2)--(3,3)--(4,3)--(4,4)--(7,4);
			
			\foreach \x in { 0,...,7}{
				\foreach \y in {0,...,4}{
					\fill[color=white] (\x,\y)circle(1.8mm); 
					
					\draw[ fill=sussexg!65](\x,\y)circle(1.2mm);
				}
			}	
			
			\foreach \x in {0,...,7}{\draw(\x, -0.3)node[below]{\small{$\x$}};}
			\foreach \x in {0,...,4}{\draw(-0.3, \x)node[left]{\small{$\x$}};}	
			
		\end{tikzpicture}
	\end{center}
	\caption{ \small  An   up-right path from $(0,0)$ to $(7,4)$ on the lattice $\Z^2$. }\label{fig:cgm} 
  \end{figure}

	A {\it bi-infinite geodesic}  is a nearest-neighbor  up-right path $\{x_k\}_{k\in\Z}$  indexed by all integers with the property that for all $m<n$, the path segment $x_{[m,n]}=\{x_k\}_{k=m}^n$ is the geodesic between $x_m$ and $x_n$. 
 A straight line   $\{x_k=x+k\evec_i\}_{k\,\in\,\Z}$, for $x\in\Z^2$ and  $i\in\{1,2\}$, is trivially a bi-infinite  geodesic because there are no alternative paths between any two of its points.  Let us call a bi-infinite geodesic {\it nontrivial} if it is not of this type.  
  The main result is that the exponential CGM has no nontrivial  bi-infinite geodesics.  
 
 \begin{theorem}\label{th:main}  Assume that weights have i.i.d.\ exponential distribution. Then with probability one, there are no nontrivial  bi-infinite geodesics.  
 \end{theorem}

 \subsection{Related work} 
 Among past work on geodesics,  
our proof is in spirit aligned with the Damron-Hanson work on FPP  \cite{damr-hans-14, damr-hans-17} and with the general LPP work in  \cite{geor-rass-sepp-17-geod, janj-rass-sepp-arxiv},   in the sense that the stationary version of the process lies at the heart of the matter.   However,  statistical properties  of the stationary versions of FPP and  of  LPP with general  weights are completely unknown.  Consequently    a straightforward adaptation of our proof to those settings is not immediately available.    

  Compared to earlier work on the exponential CGM that utilized couplings with the stationary version, such as  \cite{bala-cato-sepp, pime-16, sepp-cgm-18}, two specific   new developments made this paper  possible:
\begin{enumerate} [{(i)}]  \itemsep=3pt
\item The  discovery  in  \cite{fan-sepp-arxiv}  of the stationary distribution of the  joint LPP process with multiple characteristic directions.   A bivariate version of this distribution is constructed in Theorem \ref{th:st-lpp} below. 
\item A novel argument for controlling the location of the geodesic by coupling the bulk process with two distinct stationary processes from two different directions (Lemma \ref{lm:clo} below).  
\end{enumerate} 

One can  be fairly confident  that these features extend  to both zero-temperature and positive-temperature  polymer  models  in 1+1 dimensions that possess a tractable  stationary version.  This includes various last-passage models in both discrete and continuous space, such as those studied  in \cite{aldo-diac95, grav-trac-wido-01, joha01, oconn-yor-01, sepp97incr, sepp98perc},  and the four currently known solvable polymer models \cite{chau-noac-18-ejp}.      In positive-temperature polymer models the analogous question concerns the existence of bi-infinite Gibbs measures, as discussed in  \cite{janj-rass-19-aop-}.    These matters are left for future work. 

As in   \cite{basu-hoff-sly-arxiv-18} by Basu, Hoffman and Sly, our proof comes in two parts: 
\begin{enumerate} [{(a)}]   \itemsep=2pt 
\item The  main argument rules out bi-infinite geodesics with finite positive slope. 
\item  An easier argument shows that no geodesic can come   infinitely often arbitrarily close to an axis in the macroscopic scale.  \\[-18pt]   
\end{enumerate} 
\vspace{3pt} 
\noindent 
Beyond this superficial similarity, the   two proofs are quite different in both parts (a) and (b).  

\vspace{2pt} 

 Our  part (a)  in Section \ref{s:body}   is a  straightforward estimation of the probability that a geodesic  through the origin connects the boundaries of a square at scale $N$.  By contrast,   \cite{basu-hoff-sly-arxiv-18} controls complicated events that involve coalescence of geodesics.  This  yields additional results of interest, but the  simplicity of the bi-infinite geodesic problem is obscured.   Their sharper tools give a better estimate of the probability of a connection through the origin, namely $O(N^{-1/3})$,  while our  cruder  bound is $O(N^{-1/24})$.    In Remark \ref{rm:loss} we indicate the precise place where our estimates grow beyond  optimal order of magnitude.  

 Part (b) in  \cite{basu-hoff-sly-arxiv-18} utilizes fluctuations.  Our part (b) in Section  \ref{s:no-axis} uses the limit shape and  planarity. 
 
 \medskip 
  
 We conclude this introduction by observing that  
the non-existence of bi-infinite geodesics will  be a tool for further results. To cite an example, article \cite{janj-rass-sepp-arxiv} studies a random graph in the CGM that represents an analogue of shocks in Hamilton-Jacobi equations.  Theorem 4.3 in \cite{janj-rass-sepp-arxiv} shows that the absence of bi-infinite geodesics implies certain coalescence properties of this ``shock graph''.

 \medskip 
 
 Section \ref{s:outline} outlines the proof of Theorem \ref{th:main} and describes the organization of the rest of the paper.   We provide a self-contained exposition of the entire proof, including  proof sketches of  many auxiliary results that we use.  We collect below some notation for easy reference. 

 \subsection{Notation and conventions.} 
 	$\Z_{\ge0}=\{0,1,2,3, \dotsc\}$ and $\Z_{>0}=\{1,2,3,\dotsc\}$. 
	   For real numbers  $a$ and $b$,  $a\vee b=\max\{a,b\}$   and   $\lzb a,b  \rzb=[a,b]\cap \mathbb{Z}$.   $0$ denotes the origin of both $\R$ and $\R^2$.    $C(\e)$ and $N_0(\e)$ are  constants that  depend on a parameter $\e$ but their values   can change from line to line.  
 	
 	For $x=(x_1,x_2), y=(y_1,y_2)\in\R^2$ we use the following conventions.  
 	The standard basis vectors   are  $e_1=(1,0)$ and $e_2=(0,1)$.  The $\ell^1$-norm  is   $\abs{x}_1=\abs{x_1} + \abs{x_2}$.    
 	Integer parts  and inequalities  are interpreted coordinatewise:    $\fl{x}=(\fl{x_1}, \fl{x_2})$ and  $x\le y$  means $x_1\le y_1$ and $x_2\le y_2$.     Notation $[x,y]$ represents both   the line segment $[x,y]=\{tx+(1-t)y: 0\le t\le 1\}$   and the rectangle  $[x,y]=\{(z_1,z_2)\in\R^2:   x_i\le z_i\le  y_i \text{ for }i=1,2\}$. The context makes clear which one is used.  An open line segment is  $]x,y[\,=\{tx+(1-t)y: 0<t< 1\}$.  The lattice rectangle and  line segment are denoted  by $\lzb x,y \rzb =[x,y]\cap \mathbb{Z}^2$.   
	Path segments are  abbreviated  by $\pi_{[m,n]}  = (\pi_i)_{i=m}^n$.
 	
	 $\overline X=X-EX$ denotes a random variable $X$ centered at its mean. 
 	$X\sim$ Exp($\lambda$) for $0<\lambda<\infty$ means that the random variable $X$ has exponential distribution with rate $\lambda$, in other words $P(X>t)=e^{-\lambda t}$ for $t\ge 0$. 


 \section{Outline of the proof}\label{s:outline} 
 
 We state two auxiliary theorems and use them to prove Theorem \ref{th:main}.  Then we sketch the main ideas behind the auxiliary theorems and explain the organization of the     rest of the paper. 
 
 By the shift-invariance of the underlying weight distribution, it suffices to prove that with probability one, no nontrivial bi-infinite geodesic goes through the origin. 
 This task is split into two cases: either the geodesic ultimately stays away from the axes on a macroscopic scale, or it comes infinitely often macroscopically close to some axis.  
 
 For the first case,  
 for large positive integers $N$ and small $\e>0$, we  rule out geodesics that connect the southwest boundary of the lattice square $\lzb-N,N\rzb^2$ to its northeast boundary through the origin and whose empirical average  slope is in the range $[\e, \e^{-1}]$.   Define these portions of the boundary of the square: in the southwest
 \be\label{sw}  \partial^{N\!,\,\e}= \bigl(\, \{-N\}\times \lzb -N,-\e N \rzb\,\bigr) \cup \bigl(\, \lzb -N,-\e N\rzb \times \{-N\} \,\bigr) 
 \ee
 and in the northeast 
  \be\label{ne}  \rim{\partial}^{N\!,\,\e}=  \bigl(\,\{N\}\times \lzb \e N, N \rzb\,\bigr) \cup \bigl(\,\lzb \e N,N\rzb \times \{N\}\,\bigr) .  
 \ee

 Define the following event, illustrated in Figure \ref{fig:geod2}: 
\be\label{WNe}\begin{aligned}   W_{N\!,\,\e}
=\bigl\{ &\text{$\exists$ points $u\in\partial^{\,N\!,\,\e}$ and   $v\in\rim{\partial}^{\,N\!,\,\e}$ such that  } \\
&\qquad 
\text{the geodesic $\pi^{u,v}$ goes through the origin}  \bigr\} .  
\end{aligned}\ee  

\begin{figure}
\includegraphics[width=8cm]{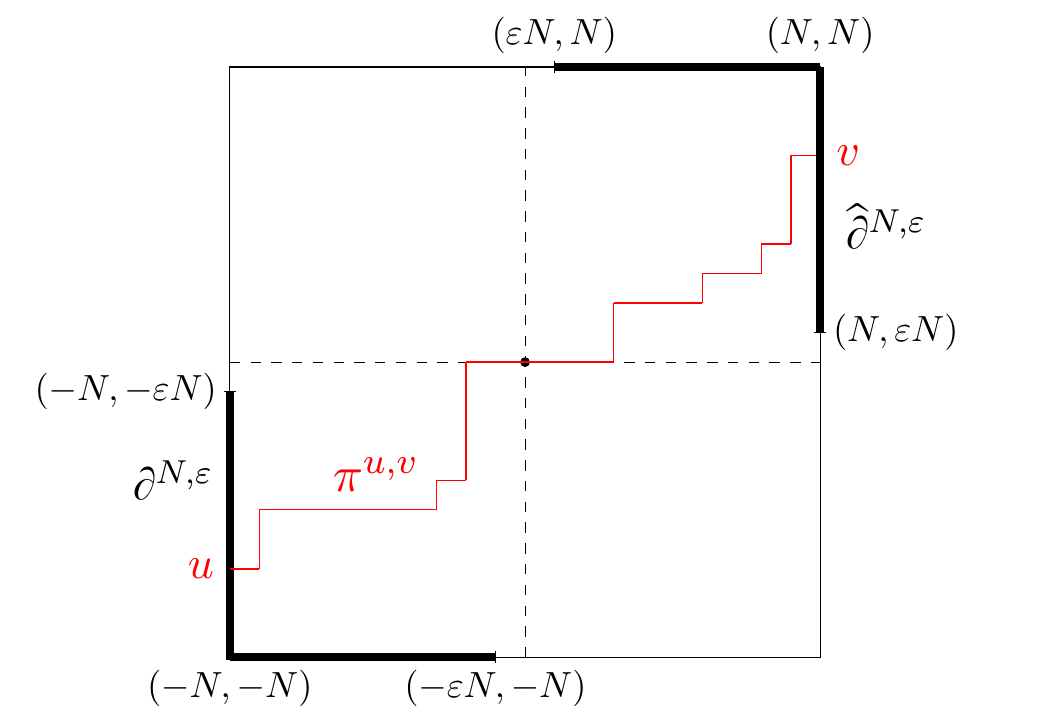}
\caption{\small The event $W_{N\!,\,\e}$. The thickset portions of the boundary  are $\partial^{N,\e}$ and $\rim{\partial}^{N,\e}$.  They are connected by the geodesic $\pi^{u,v}$   through the origin.}
\label{fig:geod2} 
\end{figure} 

We have the following quantitative control of this event. 

\begin{theorem}\label{th:ub77} 	
	For each $\e>0$ there exists a constant $C(\e)>0$ such that  $\P(W_{N\!,\,\e}) \le C(\e)N^{-\frac{1}{24}}$ for all $N\ge 1$. 
	\end{theorem}

Theorem \ref{th:ub77} rules out all geodesics that stay macroscopically away from the axes.   The next theorem shows that there are no nontrivial geodesics   that  come macroscopically arbitrarily close to an axis.

\begin{theorem} \label{th:no-axis}   The following statement  holds with probability one. 
For $i\in\{1,2\}$  and each $x\in\Z_{\ge0}^2$,  $\{x_k=x+k\evec_i\}_{k\,\in\,\Z_{\ge0}}$ is the only semi-infinite geodesic that satisfies $x_0=x$ and $\ddd\varliminf_{k\to\infty} k^{-1} x_k\cdot\evec_{3-i}=0$. 
\end{theorem} 

We combine the two theorems above  to rule out all nontrivial bi-infinite geodesics.  

\begin{proof}[Proof of Theorem \ref{th:main}, assuming Theorems \ref{th:ub77} and \ref{th:no-axis}]  
Fix a positive sequence  $\e_j\searrow0$. 
Define the event
\[  A=  \bigcap_{j\ge 1} \bigcap_{M\ge 1}  \bigcup_{N\ge M} W_{N\!,\,\e_j}^c.    \] 
Theorem \ref{th:ub77} implies that $\P(A)=1$: 
 \begin{align*}
\P(A^c)=& \P\biggl( \;\bigcup_j \bigcup_M \bigcap_{N\ge M} W_{N\!,\,\e_j}\biggr) 
 \le \sum_j  \P\biggl( \; \bigcup_M \bigcap_{N\ge M} W_{N\!,\,\e_j}\biggr)
 =  \sum_j  \lim_{M\to\infty} \P\biggl( \;   \bigcap_{N\ge M} W_{N\!,\,\e_j}\biggr)\\[4pt] 
 &\qquad\qquad
 \le   \sum_j  \lim_{M\to\infty} \P(  W_{M, \e_j})  
 \le    \sum_j  \lim_{M\to\infty} C(\e_j)  M^{-\frac{1}{24}} =0. 
 \end{align*} 

For $i\in\{1,2\}$, let  $B_i$ the event that there are no semi-infinite geodesics $\{x_k\}_{k\ge0}$ such that $x_0=0$ and $\varliminf_{k\to\infty} k^{-1} x_k\cdot\evec_{i}=0$ except for the trivial one  $\{x_k=k\evec_{3-i}\}_{k\,\in\,\Z_{\ge0}}$.   Let $R$ reflect the weight configuration across the origin:  $(R\w)_x=\w_{-x}$ for $x\in\Z^2$.  
Define the event 
\[ B=B_1\cap B_2\cap R^{-1}B_1\cap R^{-1}B_2. \] 
On the event $B$ every  semi-infinite geodesic  that either starts or ends at the origin satisfies the condition  that far enough from the origin it lies entirely inside a closed cone with apex at the origin and  disjoint from the coordinate axes.  Theorem \ref{th:no-axis} and the reflection invariance of the distribution of the weights $\w$ imply that $\P(B)=1$. 

We claim that on the full-probability event $A\cap B$ there are no nontrivial  bi-infinite geodesics   through the origin.   
 To show this, suppose  there exists a nontrivial  bi-infinite geodesic $\pi$ through the origin in the weight configuration  $\w$. 
  Consider  the following dichotomy:
\begin{enumerate}    [{(i)}]    \itemsep=4pt
\item  $\exists j,M\in\Z_{>0}$ such that   $\pi$ connects $\partial^{N\!,\,\e_j} $ to $\rim{\partial}^{N\!,\,\e_j} $ for all $N\ge M$, or  
\item   $\forall j,M\in\Z_{>0}$,   $\exists N\ge M$ such that $\pi$ misses either $\partial^{N\!,\,\e_j} $ or  $\rim{\partial}^{N\!,\,\e_j} $. \\[-12pt] 
 \end{enumerate} 
 
 Alternative (i)   forces $\w\in A^c$.   
  In alternative (ii), if  $\pi$ misses $\rim{\partial}^{N\!,\,\e_j} $ infinitely often   for each $\e_j$,  it follows that $\varliminf_{k\to\infty} k^{-1}\pi_k\cdot e_i=0$ for either $i=1$ or $2$.   Thus $\w\in B_1^c\cup B_2^c$.  Similarly,  missing $\partial^{N\!,\,\e_j} $  infinitely often   for each $\e_j$  implies  $R\w\in B_1^c\cup B_2^c$.  

 Thus a nontrivial bi-infinite geodesic through the origin is possible only on the zero-probability event $A^c\cup B^c$. 
 \end{proof} 
 
 \begin{proof}[Sketch of the proof of Theorem \ref{th:ub77}]   Theorem \ref{th:ub77} comes from two distinct  stages. 
 
 (i)  In the first stage, the southwest boundary $\partial^{N,\e}$ is divided into blocks of size $N^{2/3}$ and  the northeast boundary $\rim{\partial}^{N,\e}$ into blocks of size $N^{19/24}$.  The probability that a geodesic connects two diagonally opposite blocks through the origin  is bounded by $N^{-2/5}$ (Lemma \ref{lm:clo}).  The control here comes from  random walk bounds on the location where a geodesic crosses the $y$-axis.  These  bounds are developed through a coupling with   increment-stationary LPP processes. 
 
 (ii) The second stage shows that any geodesic that connects an $N^{2/3}$-block through the origin to a point {\it outside}  its opposite $N^{19/24}$-block violates the  $N^{2/3}$ KPZ wandering exponent.  Through another coupling argument, the probability of this happening is  bounded  by $N^{-3/8}$ (Lemma \ref{lm:far}).   
 
Multiplying by the number of $N^{2/3}$-blocks gives  the estimate $O(N^{1/3}\cdot N^{-2/5}+N^{1/3}\cdot N^{-3/8})= O(N^{-1/24})$. 
 \end{proof}
 
 \begin{proof}[Sketch of the proof of Theorem \ref{th:no-axis}]    
 Comparison with  increment-stationary LPP processes shows that the quantity $G_{0,\pi_n}-G_{e_2,\pi_n}$ blows up if $\pi_n$ is a path above the $x$-axis but  $n^{-1}\pi_n$ comes arbitrarily close to the $x$-axis.   This rules out the possibility that $\pi_\bbullet$ is a geodesic. 
 \end{proof}

The next two sections develop tools: Section \ref{s:statLPP} a coupling of increment-stationary LPP processes and Section \ref{s:bound} bounds on geodesic fluctuations. The  proof of Theorem \ref{th:ub77}  follows   in Section \ref{s:body} and that   of  Theorem  \ref{th:no-axis}   in Section \ref{s:no-axis}.  

\section{Stationary last-passage percolation} 	\label{s:statLPP}

Pick $0<\lambda<\rho<1$ and a base vertex $u\in\Z^2$.  We construct   two coupled LPP processes  $G^\lambda_{u,\bbullet}$ and $G^\rho_{u,\bbullet}$   on the nonnegative  quadrant $u+\Z_{\ge0}^2$  such that their increments are jointly stationary under lattice translations.  
Both processes use the same   i.i.d.\ Exp(1) weights $\{\w_x\}_{x\,\in\,u+\Z_{>0}^2}$ in the bulk.  They have  boundary conditions on the positive  $x$- and $y$-axes centered at $u$, coupled in a way described in the next theorem. 

For $\alpha\in\{\lambda, \rho\}$, the definition of the process $G^\alpha_{u,\bbullet}$  goes as follows.  
 The  boundary weights are denoted by  
	$\{I^\alpha_{u+ie_1}, J^\alpha_{u+je_2}:i, j\in\Z_{>0}\}$.   Put  $G^\alpha_{u,u}=0$ and on the boundaries 
	\be\label{Gr11a} G^\alpha_{u,\,u+\,ke_1}=\sum_{i=1}^k I^\alpha_{ie_1} 
	\quad\text{and}\quad
	G^\alpha_{u,\,u+\,le_2}= \sum_{j=1}^l  J^\alpha_{je_2}  \quad\text{ for } k,l\ge 1.   \ee 
	In the bulk 
	for $x=(x_1,x_2)\in u+ \Z_{>0}^2$, 
	\be\label{Gr12a}\begin{aligned} 
	G^\alpha_{u,\,x}&= \max_{1\le k\le x_1-u_1} \;  \Bigl\{  \;\sum_{i=1}^k I^\alpha_{u+ie_1}  + G_{u+ke_1+e_2, \,x} \Bigr\}  
	\bigvee
	\max_{1\le \ell\le x_2-u_2}\; \Bigl\{  \;\sum_{j=1}^\ell  J^\alpha_{u+je_2}  + G_{u+e_1+\ell e_2, \,x} \Bigr\} \\
	&=G^\alpha_{u,\,x-e_1}\vee G^\alpha_{u,\,x-e_2} + \w_x.
	\end{aligned} \ee
	 $G^\alpha_{u,\bbullet}$ does not use a weight at the base point $u$.  Above $G_{x,y}$ is the LPP process \eqref{v:G} that uses the bulk weights $\w$. 
Define increment variables  for vertices $x\in u+\Z_{>0}^2$  by 
\be\label{IJ85}    I^\alpha_x=G^\alpha_x-G^\alpha_{x-e_1} 
\quad\text{and}\quad 
J^\alpha_x=G^\alpha_x-G^\alpha_{x-e_2}. 
\ee
An important part of the next theorem for the sequel is the independence of various collections of increment variables.  These are illustrated in Figure \ref{fig:indIJ}. 
 
 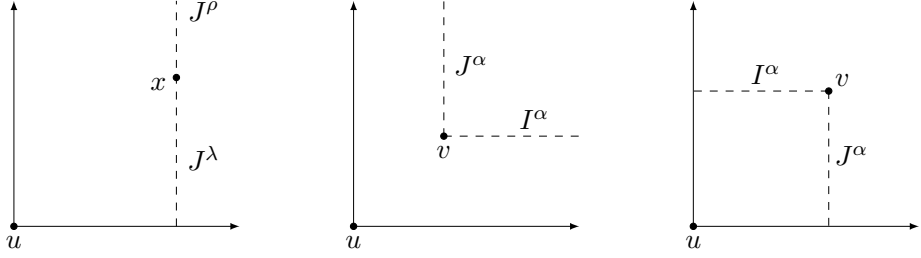
\begin{figure}
	\begin{tikzpicture}[scale=0.6, every node/.style={transform shape}]
	\def\s{1.6}
	\draw[->,line width=0.08mm,>=latex]   (0,0) --(0,5); 
	\draw[->,line width=0.08mm,>=latex]   (0,0) --(5,0); 
	\node [scale=\s][below] at (0,0) {$u$};
	\draw [fill] (0,0) circle [radius=0.07];
	\node [scale=\s][below] at (3.2,3.5) {$x$};
	\draw [dashed][line width=0.01cm] (3.6,0) -- (3.6,5);
	\node [scale=\s][below] at (4.2,2) {$J^\lambda$};
	\node [scale=\s][below] at (4.2,5.2) {$J^\rho$};
	\draw [fill] (3.6,3.3) circle [radius=0.07];
	\end{tikzpicture}
	\qquad\quad 
	\begin{tikzpicture}[scale=0.6, every node/.style={transform shape}]
	\def\s{1.6}
	\draw[->,line width=0.08mm,>=latex]   (0,0) --(0,5); 
	\draw[->,line width=0.08mm,>=latex]   (0,0) --(5,0); 
	\node [scale=\s][below] at (0,0) {$u$};
	\draw [fill] (0,0) circle [radius=0.07];
	\node [scale=\s][below] at (2,2) {$v$};
	\draw [fill] (2,2) circle [radius=0.07];
	\draw [dashed][line width=0.01cm] (2,2) -- (2,5);
	\draw [dashed][line width=0.01cm] (2,2) -- (5,2);
	\node [scale=\s][below] at (4,2.8) {$I^\alpha$};
	\node [scale=\s][below] at (2.6,4) {$J^\alpha$};
	\end{tikzpicture}
	\qquad\quad  
	\begin{tikzpicture}[scale=0.6, every node/.style={transform shape}]
	\def\s{1.6}
\draw[->,line width=0.08mm,>=latex]   (0,0) --(0,5); 
\draw[->,line width=0.08mm,>=latex]   (0,0) --(5,0); 
\node [scale=\s][below] at (0,0) {$u$};
\draw [fill] (0,0) circle [radius=0.07];
\node [scale=\s][below] at (3.3,3.6) {$v$};
\draw [fill] (3,3) circle [radius=0.07];
\draw [dashed][line width=0.01cm] (3,3) -- (0,3);
\draw [dashed][line width=0.01cm] (3,3) -- (3,0);
\node [scale=\s][below] at (1.6,3.8) {$I^\alpha$};
\node [scale=\s][below] at (3.5,2) {$J^\alpha$};
	\end{tikzpicture}
	\caption{\small The independent increment variables from Theorem \ref{th:st-lpp}. Left: $J^\lambda$ below $x$ and $J^\rho$ above $x$ from part (i). Middle and right: $I^\alpha$ and $J^\alpha$ increments on down-right lattice paths from part (ii).}
	\label{fig:indIJ} 
\end{figure}

\begin{theorem}\label{th:st-lpp}   Let  $0<\lambda<\rho<1$ and $u\in\Z^2$.  There exists a coupling of the boundary weights $\{I^\lambda_{u+ie_1},I^\rho_{u+ie_1}, 	J^\lambda_{u+je_2}$, $J^\rho_{u+je_2}:i, j\in\Z_{>0}\}$  such that the joint process $(G^\lambda_{u,\bbullet}\,, G^\rho_{u,\bbullet})$ has the following properties.  
\begin{enumerate}[{\rm(i)}] \itemsep=5pt 
\item {\rm(Joint)}     The joint process of increments is stationary: for each $v\in u+\Z_{\ge0}^2$,  
\be\label{Gr17}    \bigl\{(G^\lambda_{u,v+x}-G^\lambda_{u,v}, G^\rho_{u,v+x}-G^\rho_{u,v}): x\in\Z_{\ge0}^2\bigr\}
\deq
 \bigl\{(G^\lambda_{u,u+x}, G^\rho_{u, u+x}): x\in\Z_{\ge0}^2\bigr\}.
 \ee
The following independence property holds along vertical lines: for each $x\in u+\Z_{>0}^2$, the variables $\{ J^\lambda_{x+je_2}: u_2-x_2+1\le j\le 0\}$ and $\{ J^\rho_{x+je_2}: j\ge 1\}$ are mutually independent.  

\item {\rm(Marginal)}  For both $\alpha\in\{\lambda, \rho\}$ and for each $v\in u+\Z_{\ge0}^2$,  the increment variables  $\{I^\alpha_{v+ie_1}, J^\alpha_{v+je_2}:i, j\in\Z_{>0}\}$	 are mutually independent with marginal distributions 
\be\label{Gr23} 
I^\alpha_{v+ie_1}\sim\text{\rm Exp}(1-\alpha) \quad \text{ and }\quad J^\alpha_{v+je_2}\sim\text{\rm Exp}(\alpha). 
\ee
The same is true for the variables $\{I^\alpha_{v-ie_1}, J^\alpha_{v-je_2}: 0\le i <  v_1-u_1,\, 0\le j< v_2-u_2\}$. 	
\end{enumerate} 
\end{theorem} 

\begin{proof} 
We  construct  a  joint LPP process $(L^\lambda_{x}, L^\rho_{x})_{x\,\in\,u+\Z_{\ge0}\times\Z}$ on the discrete  right half-plane with origin at $u$.  
First define the  boundary weights  $\Jvec^\lambda=\{J^\lambda_{u+je_2}\}_{j\in\Z}$ and  $\Jvec^\rho=\{J^\rho_{u+je_2}\}_{j\in\Z}$ 
on the $y$-axis centered at $u$.    For $\alpha\in\{\lambda, \rho\}$ let $\Yvec^\lambda=\{Y^\lambda_j\}_{j\in\Z}$ and $\Yvec^\rho=\{Y^\rho_j\}_{j\in\Z}$   be independent  sequences of  i.i.d.\ variables with marginal distributions  $Y^\alpha_j\sim$ Exp$(\alpha)$. Then define $(\Jvec^\rho, \Jvec^\lambda)=(\Yvec^\rho, \Dop(\Yvec^\lambda, \Yvec^\rho))$ in terms of the departure process operator $\Dop$  from \eqref{DSR} in Appendix \ref{app:queues}.   This gives coupled sequences $(\Jvec^\rho, \Jvec^\lambda)$.  

  For $\alpha\in\{\lambda, \rho\}$ define the LPP values on the $y$-axis  by 
\be\label{Ly}  
L^\alpha_{u}=0, \quad  L^\alpha_{u+je_2}- L^\alpha_{u+(j-1)e_2} = J^\alpha_{u+je_2}\quad \text{for } j\in\Z. 
 \ee
This results in negative values  $L^\alpha_{u+je_2}$ for $j<0$.  Complete the definitions by putting, again for $\alpha\in\{\lambda, \rho\}$ and now for $x\in u+\Z_{>0}\times\Z$, 
\be\label{Ly4}  L^\alpha_{x}=\sup_{j: j\le x_2-u_2} \bigl\{  L^\alpha_{u+je_2} + G_{u+e_1+je_2, x}\bigr\}, \quad 
I^\alpha_x=L^\alpha_{x}-L^\alpha_{x-e_1} 
\quad\text{and}\quad
J^\alpha_x=L^\alpha_{x}-L^\alpha_{x-e_2}. 
\ee

For $k\ge 0$, denote the sequences of $J$-increments  on the vertical line shifted by $ke_1$ from the $y$-axis by 
$\Jvec^{\alpha,k}=\{J^{\alpha,k}_j\}_{j\in\Z}=\{J^\alpha_{u+ke_1+je_2}\}_{j\in\Z}$ and the sequences of weights by $\servv^k=\{\serv^k_j\}_{j\in\Z}=\{\w_{u+ke_1+je_2}\}_{j\in\Z}$.    $\Jvec^{\alpha,0}$ is the original boundary sequence $\Jvec^\alpha$ we began with.   Then in terms of Lemma \ref{lm:DR5} we have the following. With $(\sigma, \alpha_1, \alpha_2)=(1,\rho, \lambda)$,   $( \Jvec^\rho, \Jvec^\lambda)$ has the distribution of $(\arrv^1, \arrv^2)$ and for each $k\ge 1$ and $\alpha\in\{\lambda, \rho\}$,  $\Jvec^{\alpha,k}=\Dop( \Jvec^{\alpha,k-1}, \servv^k)$.  Repeated application of Lemma \ref{lm:DR5} implies the distributional equality  $( \Jvec^{\rho,k},  \Jvec^{\lambda,k})\deq( \Jvec^\rho, \Jvec^\lambda)$  for all $k\ge 0$.  

The evolution in \eqref{Ly4} satisfies a semigroup property:  
for each $k$ the values $L^\alpha_x$   for $x_1\ge u_1+k+1$ satisfy 
\[    L^\alpha_{x}=\sup_{j: j\le x_2-u_2} \bigl\{  L^\alpha_{u+ke_1+je_2} + G_{u+(k+1)e_1+je_2, x}\bigr\}. 
\] 
 It follows that the entire process of increments is invariant under translations that keep it in the half-space:  for $z\in \Z_{\ge0}\times\Z$, 
\be\label{L55} \begin{aligned}  
&\{ I^\lambda_{z+x+e_1}, I^\rho_{z+x+e_1}, J^\lambda_{z+x}, J^\rho_{z+x}: x\in u+\Z_{\ge0}\times\Z\}   \\
&\qquad 
\deq
\{ I^\lambda_{x+e_1}, I^\rho_{x+e_1}, J^\lambda_x, J^\rho_x: x\in u+\Z_{\ge0}\times\Z\}. 
\end{aligned}\ee
(The index is $x+e_1$ rather than $x$ in the $I$-increments simply because these are not defined on the boundary where $x_1=u_1$.) 

We claim that for $\alpha\in\{\lambda, \rho\}$  and for any new base point  $v\in u+\Z_{\ge0}\times\Z$,
\be\label{L57} \begin{aligned} 
  &\text{$\{I^\alpha_{v+ie_1}, J^\alpha_{v+je_2}:i, j\in\Z_{>0}\}$	are mutually independent with marginal distributions}\\
   &\qquad 
   I^\alpha_{v+ie_1}\sim\text{\rm Exp}(1-\alpha)\quad \text{ and } \quad  J^\alpha_{v+je_2}\sim\text{\rm Exp}(\alpha).
\end{aligned} \ee
Since everything is shift-invariant, we can take  $v=u$.  As observed  above,  $\Jvec^\alpha$ is a sequence of i.i.d.\ Exp$(\alpha)$ random variables  by Lemma \ref{lm:DR5}(i). Thus it suffices to prove the marginal statement about $\{I^\alpha_{u+ie_1}:i\ge1\}$ because these variables are a function of $\{J^\alpha_{u+je_2},\, \w_{(i,j)}:i\ge1,  j\le0\}$ which are independent of $\{ J^\alpha_{u+je_2}: j\ge1\}$. 

The claim for $\{I^\alpha_{u+ie_1}:i\ge1\}$ follows from proving inductively the following statement for each $n\ge 1$: 
\be\label{L59} \begin{aligned} 
  &\text{$\{I^\alpha_{u+ie_1}, J^\alpha_{u+ne_1+je_2}: 1\le i\le n, j\le0\}$	are mutually independent with}\\
  & \text{marginal distributions}\quad 
   I^\alpha_{u+ie_1}\sim\text{\rm Exp}(1-\alpha)\quad \text{ and } \quad  J^\alpha_{u+ne_1+je_2}\sim\text{\rm Exp}(\alpha).
\end{aligned} \ee
This claim  is a consequence of Lemma \ref{lm:DR5}(ii).  Begin with the case $n=1$.  The inputs are now inter-arrival times 
$\{\arr_j=J^\alpha_{u+je_2}:j\le0\}$ and service times $\{\serv_j=\w_{(1,j)}: j\le 0\}$, out which we compute the inter-departure times $\{\depa_j=J^\alpha_{u+e_1+je_2}:j\le0\}$ and the sojourn time $\sojo_0=I^\alpha_{u+e_1}$.   Continue inductively.  Assume that \eqref{L59} holds for a given $n$. Then feed to the queueing operators inter-arrival times 
$\{\arr_j=J^\alpha_{u+ne_1+je_2}:j\le0\}$ and service times $\{\serv_j=\w_{(n+1,j)}: j\le 0\}$, all independent of $\{I^\alpha_{u+ie_1}: 1\le i\le n\}$.   Compute the inter-departure times $\{\depa_j=J^\alpha_{u+(n+1)e_1+je_2}:j\le0\}$ and the sojourn time $\sojo_0=I^\alpha_{u+(n+1)e_1}$.    Lemma \ref{lm:DR5}(ii) extends the validity of \eqref{L59} to  $n+1$.  Claim \eqref{L57} has been verified. 

To prove Theorem \ref{th:st-lpp}, take the coupled boundary weights 
$\{ I^\alpha_{u+ie_1}, J^\alpha_{u+je_2}:  i,j\ge 1, \alpha\in\{\lambda, \rho\}\}$ as constructed above.   The LPP process $\{G^\alpha_{u,x}: x\in u+\Z_{\ge0}^2\}$ defined by \eqref{Gr11a}--\eqref{Gr12a} is then exactly the same as the restriction  $\{L^\alpha_x: x\in u+\Z_{\ge0}^2\}$   of $L^\alpha$.  Namely, \eqref{Gr12a} can be rewritten as follows:
\begin{align*}
G^\alpha_{u,\,x}&= \max_{1\le k\le x_1-u_1} \;  \bigl\{  L^\alpha_{u+ke_1}   + G_{u+ke_1+e_2, \,x} \bigr\}  
	\bigvee
	\max_{1\le \ell\le x_2-u_2}\; \bigl\{  L^\alpha_{u+\ell e_2}  + G_{u+e_1+\ell e_2, \,x} \bigr\}\\
	&= \sup_{j\le 0}   \;  \bigl\{  L^\alpha_{u+je_2}  +   \max_{1\le k\le x_1-u_1} \bigl[  G_{u+e_1+je_2, u+ke_1}   + G_{u+ke_1+e_2, \,x}  \bigr] \,\bigr\}  \\
	&\qquad\qquad \qquad 
	\bigvee
	\max_{1\le \ell\le x_2-u_2}\; \bigl\{  L^\alpha_{u+\ell e_2}  + G_{u+e_1+\ell e_2, \,x} \bigr\} \\
	&=   \sup_{j: j\le x_2-u_2} \bigl\{  L^\alpha_{u+je_2} + G_{u+e_1+je_2, x}\bigr\} 
	 \; =\;  L^\alpha_x. 
\end{align*}
Invariance \eqref{Gr17} comes from \eqref{L55}.  
 The first statement of part (ii) of the theorem comes from \eqref{L57}, the second statement from   \eqref{L59}. 
\end{proof}

\section{Bounds for geodesic fluctuations} \label{s:bound}

Let $G^\rho_{u, \bbullet}$ be a stationary LPP process with base point $u$ as described in Theorem \ref{th:st-lpp}, with independent boundary weights $I_{u+ie_1}\sim$ Exp$(1-\rho)$ and $J_{u+je_2}\sim$ Exp$(\rho)$ for $i,j\ge1$.  
	For a  northeast endpoint  $x\in u+\Z_{>0}^2$, let $Z^\rho_{u,x}$  be the signed exit point of the geodesic $\pi^{\rho,u,x}_\bbullet$ of $G^\rho_{u,x}$ from the west and south boundaries  of $u+\Z_{\ge0}^2$. More precisely,
	\be\label{Zdef} 
	Z^\rho_{u,x}=
	\begin{cases}
	\argmax{k} \bigl\{ \,\sum_{i=1}^k I_{u+ie_1}  + G_{u+ke_1+e_2, \,x} \bigr\},  &\text{if } \pi^{\rho,u,x}_1=u+e_1,\\
	-\argmax{\ell}\bigl\{  \;\sum_{j=1}^\ell  J_{u+je_2}  + G_{u+\ell e_2+e_1, \,x} \bigr\},  &\text{if } \pi^{\rho,u,x}_1=u+e_2.
	\end{cases}
	\ee 

The open line segment of interior directions is denoted by $]e_2, e_1[=\{(s, 1-s): 0<s<1\}$.   The parameter $\rho\in(0,1)$ of the stationary LPP process is in one-to-one correspondence with a direction vector $\xi=(\xi_1, 1-\xi_1)\in\,]e_2, e_1[$ through these equations: 
\be \label{xi-rho} 
\xi=\xi(\rho)  
=\left(\frac{(1-\rho)^2}{(1-\rho)^2+\rho^2}\,,\frac{\rho^2}{(1-\rho)^2+\rho^2}\right)
\ \iff\ 
\rho=\rho(\xi) 
=\frac{\sqrt{1-\xi_1}}{\sqrt{\xi_1}+\sqrt{1-\xi_1}}. 
\ee
Direction  $\xi(\rho)$ is called  the  {\it  characteristic direction}   associated to the parameter $\rho$.  A key property  that distinguishes $\xi(\rho)$  among all    $\eta\in\,]e_2, e_1[$ is that  $\abs{Z^\rho_{u,u+\fl{N\eta}}}=o(N)$ almost surely  if and only if $\eta=\xi(\rho)$. 
Write the characteristic direction 
as 
\[  \xi(\rho)=(\xi_1(\rho), \xi_2(\rho))=\aaa{\rho}((1-\rho)^2, \rho^2) \]  
by introducing  
\be\label{a-xi}
 \aaa{\rho}=  \frac1{(1-\rho)^2+ \rho^2} .\ee
Note the bounds  $1\le \aaa{\rho}\le 2$.     

This section derives basic estimates  for later use.   We take the base point as the origin $u=0$ but in later applications the base point will vary.  
Abbreviate the sum of boundary weights on the $x$-axis as  $S^\rho_k=\sum_{i=1}^k I^\rho_{ie_1}=G^\rho_{0,ke_1}$.   The starting point for the estimates is the variance formula of the next theorem. 

\begin{theorem} \label{thGr3} 
For $(m,n)\in\Z_{>0}^2$, 
\begin{align}\label{VGr}
\Vvv[G^\rho_{0,(m,n)}] &=  -\,\frac{m}{(1-\rho)^2}+\frac{n}{\rho^2} +\frac{2}{1-\rho} \, \E
\bigl[ S^\rho_{(Z^\rho_{0, (m,n)})^+}\bigr]  . 
\end{align} 
\end{theorem}

 \begin{proof}[Sketch of proof]  We give the main steps of the argument. Detailed proofs appear  in Lemma 4.6 of \cite{bala-cato-sepp} and in Section 5.3 of \cite{sepp-cgm-18}. 
 Utilizing 
 \[G^\rho_{0,(m,n)}
 =\sum_{i=1}^m   I^\rho_{(i,0)} +\sum_{j=1}^n J^\rho_{(m,j)} 
 =\sum_{j=1}^n J^\rho_{(0,j)} + \sum_{i=1}^m   I^\rho_{(i,n)}  
 \]
 and  the independence of  $\{ I^\rho_{(i,n)}, J^\rho_{(m,j)}: 1\le i\le m, 1\le j\le n\}$ from Theorem \ref{th:st-lpp}(ii), deduce 
\be\begin{aligned}
\Vvv\bigl[G^\rho_{0,(m,n)}\bigr]
 &=-\Vvv\biggl[\, \sum_{i=1}^m I^\rho_{(i,n)}\biggr] +  \Vvv\biggl[\, \sum_{j=1}^n J^\rho_{(0,j)}\biggr] +2\,\Cvv\biggl[\, \sum_{i=1}^m   I^\rho_{(i,0)} \,, \sum_{i=1}^m I^\rho_{(i,n)} \biggr] . 
\end{aligned}\label{aux3}\ee
The first two terms of \eqref{VGr} and \eqref{aux3} match.  Let $I^{\lambda,\rho}_x$ be increment variables \eqref{IJ85} for a process whose independent boundary weights satisfy  $I^{\lambda,\rho}_{(i,0)}\sim$ Exp$(\lambda)$   and $J^{\lambda,\rho}_{(0,j)}\sim$ Exp$(\rho)$.  Complete  the proof  through 
\[  \Cvv\biggl[\, \sum_{i=1}^m   I^\rho_{(i,0)} \,, \sum_{i=1}^m I^\rho_{(i,n)} \biggr] 
= - \frac{\partial}{\partial \lambda} \E\biggl[\, \sum_{i=1}^m I^{\lambda,\rho}_{(i,n)}\biggr] \bigg\vert_{\lambda=1-\rho} 
= \frac{1}{1-\rho} \, \E\bigl[ S^\rho_{(Z^\rho_{0, (m,n)})^+}\bigr]  . 
\] 
The line above comes by calculating the middle derivative in two ways.  For the left equality, 
 condition on $\sum_{i=1}^m I^{\lambda,\rho}_{(i,0)}$ and  differentiate its density.  For the right equality, express the boundary variables $I^{\lambda,\rho}_{(i,0)}$ as functions of uniform random variables and take the differentiation inside the expectation. 
 \end{proof}
 
 Next a bound on the exit point.  This CGM  result  is from \cite{bala-cato-sepp} that adapted  the seminal result from \cite{cato-groe-06}.   A proof appears also in  Section 5.4 of \cite{sepp-cgm-18}. 
 
\begin{theorem} \label{tpr6}  
For  $0<\e<\tfrac12$ and $\kappa>0$  
     there exists a finite constant  $B(\e, \kappa)$ such that   
\be\label{t8}  \P\bigl\{\abs{Z^\rho_{0,(m,n)}}\ge \ell\bigr\} \le   B(\e, \kappa)\Bigl(\,\frac{N^2}{\ell^3}  +    \frac{N^{8/3}}{\ell^4}\Bigr)
\qquad 
\text{ for  all $m,n, N, \ell \ge 1$ }  \ee
whenever $\rho\in[\e, 1-\e]$ and  $\abs{(m,n)-N\xi(\rho)}_1 \le \kappa$.
\end{theorem}

\begin{proof}    It suffices to prove the bound 
\be\label{t11}  \P\bigl\{Z^\rho_{0,(m,n)}\ge \ell\bigr\} \le   B(\e, \kappa)\Bigl(\,\frac{N^2}{\ell^3}  +    \frac{N^{8/3}}{\ell^4}\Bigr) \ee
because the other probability $\P\{Z^\rho_{0,(m,n)}\le -\ell\}$ is  obtained by reflection across the diagonal.  
    We can assume that  $\ell\le m$ for otherwise the probability in \eqref{t11} vanishes. 
Let $0<r<1$ be a constant that will be set small enough in the proof.   Let 
\be\label{la1}
\lambda=\rho+\frac{r\ell}N. 
\ee
We take $r=r(\e, \kappa)$ at least small enough so that $rm/N<\tfrac12(1-\rho)$ for $m\le N(1-\rho)^2+\kappa$ and $N\ge 1$.  This guarantees  that  $\lambda\in(\rho, \frac{1+\rho}2)$ is also a legitimate parameter for an increment-stationary CGM. 

Couple the boundary weights so that $I^\lambda_{ie_1}\ge I^\rho_{ie_1}$.  
In the first inequality  below  use  $S^\lambda_{k}+G_{(k,1),(m,n)}\le G^\lambda_{0,(m,n)}$.  The second equality follows from $I^\lambda_{ie_1}\ge I^\rho_{ie_1}$.      Recall that $\overline X=X-\E X$.  
\begin{align}
&\P\{ Z^\rho_{0,(m,n)}\ge \ell\} =\P\{\,\exists k\ge \ell: \, S^\rho_{k}+G_{(k,1),(m,n)}=G^\rho_{0,(m,n)} \,\} \nn\\
&\quad\le  \P\{\,\exists k\ge \ell: \,  S^\lambda_{k}-S^\rho_{k}\le G^\lambda_{0,(m,n)} -G^\rho_{0,(m,n)} \,\} \nn\\
&\quad =  \P\{\,  S^\lambda_{\ell}-S^\rho_{\ell}\le G^\lambda_{0,(m,n)} -G^\rho_{0,(m,n)} \,\} \nn\\
&\quad =\P\Bigl\{\,  \overline{S^\lambda_{\ell}}- \overline{S^\rho_{\ell}}  \le  \overline{G^\lambda_{0,(m,n)}} -\overline{G^\rho_{0,(m,n)}}  -  \bigl( \,\E[S^\lambda_{\ell}-S^\rho_{\ell}] -  \E[G^\lambda_{0,(m,n)}-G^\rho_{0,(m,n)}] \,\bigr)    \,\Bigr\} . 
\label{L78}
\end{align}

Compute and bound  the means   in the last probability above.    
\begin{align}\label{65} 
&\E[S^\lambda_{\ell}-S^\rho_{\ell}]=  \ell\Bigl(\frac1{1-\lambda}-\frac1{1-\rho}\Bigr)
= \frac{\ell}{(1-\lambda)(1-\rho)} (\lambda-\rho)
= \frac{1}{(1-\lambda)(1-\rho)} \cdot  \frac{r\ell^2}{N}
\end{align} 
Introduce the quantities 
$\kappa^1_N=m-N\xi_1(\rho)$ and $\kappa^2_N=n-N\xi_2(\rho)$
that  satisfy  
$  \abs{\kappa^1_N}+\abs{\kappa^2_N}\le \kappa.  $  
Then for  the means of the LPP values, 
\be\label{66} \begin{aligned}
&\E[G^\lambda_{0,(m,n)}-G^\rho_{0,(m,n)}]  
=m\Bigl(\frac1{1-\lambda}-\frac1{1-\rho}\Bigr) + n\Bigl(\frac1{\lambda}-\frac1{\rho}\Bigr)\\
&\quad 
=\Bigl( \,\frac{m}{(1-\lambda)(1-\rho)}-\frac{n}{\lambda\rho}\,\Bigr)(\lambda-\rho) \\
&\quad 
=\aaa{\rho}  N \Bigl( \,\frac{1-\rho}{1-\lambda}-\frac{\rho}{\lambda}\,\Bigr)(\lambda-\rho)
+ \Bigl( \frac{\kappa^1_N}{(1-\lambda)(1-\rho)}-\frac{\kappa^2_N}{\lambda\rho}\Bigr)(\lambda-\rho)  \\
&\quad 
=  \frac{\aaa{\rho}N}{\lambda(1-\lambda)} (\lambda-\rho)^2
+ \Bigl( \frac{\kappa^1_N}{(1-\lambda)(1-\rho)}-\frac{\kappa^2_N}{\lambda\rho}\Bigr)(\lambda-\rho)  \\
&\quad 
=  \frac{\aaa{\rho}r^2\ell^2}{\lambda(1-\lambda)N}  
+ \Bigl( \frac{\kappa^1_N}{(1-\lambda)(1-\rho)}-\frac{\kappa^2_N}{\lambda\rho}\Bigr)\frac{r\ell}N  \\
&\quad 
\le   \frac{\aaa{\rho}r^2\ell^2}{\lambda(1-\lambda)N}  +   C_1(\e, \kappa)   \frac{r\ell}{N} .
\end{aligned}\ee
 

Comparison of   \eqref{65} and  \eqref{66}  shows that if we choose   $r$ and $c_3$ small enough  as functions of $(\e, \kappa)$, then there is a constant $\ell_0(\e, \kappa)\ge 1$ such that for $\ell\ge \ell_0(\e, \kappa)$ and $\rho\in[\e, 1-\e]$ we have 
\be\label{68}  \E[S^\lambda_{\ell}-S^\rho_{\ell}]  > \E[G^\lambda_{0,(m,n)}-G^\rho_{0,(m,n)}]   +    c_3\frac{r\ell^2}{N}. \ee 
 
We continue the bound on $\P\{ Z^\rho_{0,(m,n)}\ge \ell\} $  from line \eqref{L78} and apply \eqref{68}.    Below we pack the $(\e, \kappa)$-dependent factors 
  into a constant $C=C(\e, \kappa)$.  
\begin{align}
&\P\{ Z^\rho_{0,(m,n)}\ge \ell\}  \le  
\P\Bigl\{\,  \overline{S^\lambda_{\ell}}-\overline{S^\rho_{\ell}}  \le  \overline{G^\lambda_{0,(m,n)}} -\overline{G^\rho_{0,(m,n)}}  -  c_3\frac{r\ell^2}{N}  \,\Bigr\} \nn\\
&\quad\le 
\P\Bigl\{\,  \overline{S^\lambda_{\ell}}-\overline{S^\rho_{\ell}}  \le   -  c_3\frac{r\ell^2}{2N}  \,\Bigr\}
+
\P\Bigl\{\,    \overline{G^\lambda_{0,(m,n)}} -\overline{G^\rho_{0,(m,n)}}  \ge  c_3\frac{r\ell^2}{2N}  \,\Bigr\} \nn\\
&\quad\le 
\frac{CN^2}{\ell^4} \Vvv[S^\lambda_{\ell}-S^\rho_{\ell}] 
+   \frac{CN^2}{\ell^4} \Vvv[G^\lambda_{0,(m,n)} -G^\rho_{0,(m,n)}]  
\nn\\
&\quad\le 
\frac{CN^2}{\ell^3}  +   \frac{CN^2}{\ell^4} \bigl(\, \Vvv[G^\lambda_{0,(m,n)} ] +\Vvv[G^\rho_{0,(m,n)}] \, \bigr) \nn\\
&\quad\le 
\frac{CN^2}{\ell^3}  +   \frac{CN^2}{\ell^4}  \bigl( \Vvv[G^\rho_{0,(m,n)}]  +   m(\lambda-\rho)\bigr) \nn\\
&\quad= 
\frac{CN^2}{\ell^3}  +   \frac{CN^2}{\ell^4}  \Bigl(   -\,\frac{m}{(1-\rho)^2}+\frac{n}{\rho^2} +\frac{2}{1-\rho}  \E[S^\rho_{Z^\rho_{0,(m,n)}}]  +  ((1-\rho)^2N+\kappa)  \cdot 
\frac{r\ell}N\,\Bigr) \nn\\
&\quad\le  
\frac{CN^2}{\ell^3}  +   \frac{CN^2}{\ell^4}  \bigl(     \E[Z^\rho_{0,(m,n)}]  +   \ell \bigr) 
\label{L82} 
\le   
\frac{CN^2}{\ell^3}  +   \frac{CN^2}{\ell^4}      \E[Z^\rho_{0,(m,n)}] . 
\end{align}

 Along the way we used the following  two inequalities.  For $\e\le\rho\le\lambda\le 1-\e/2$, 
 \[  \Vvv[G^\lambda_{0,(m,n)}] \le    \Vvv[G^\rho_{0,(m,n)}]
 + Cm(\lambda-\rho)
 \]
 holds  by the variance formula \eqref{VGr}
 \cite[Lemma 5.7]{sepp-cgm-18}.  Next,  even though  the i.i.d.\ terms $I^\rho_{ie_1}$ are positively correlated with $Z^\rho_{0,(m,n)}$,  we have the bound 
 \[  \E[S^\rho_{Z^\rho_{0,(m,n)}}] \le C\bigl(  \E[Z^\rho_{0,(m,n)}] + 1\bigr) \]
  because the terms   $I^\rho_{ie_1}$ have high  moments   
\cite[Lemma 5.8]{sepp-cgm-18}.
 
  Define a constant  $b=\ell_0+C$ 
 with $\ell_0(\e,\kappa)$   determined above \eqref{68}   and   $C(\e,\kappa)$ from line \eqref{L82}  above.     Then 
\begin{align*}
  \E[Z^\rho_{0,(m,n)}]  &= \int_0^m  \P(Z^\rho_{0,(m,n)}\ge s) \,ds 
\;  \le \; bN^{2/3} + 
 C \int_{bN^{2/3}}^\infty  \Bigl(   \frac{N^2}{s^3}  +   \frac{N^2}{s^4}      \E[Z^\rho_{0,(m,n)}] \Bigr)\,ds
\\
&= bN^{2/3} +   \frac{CN^{2/3}}{2b^2} +   \frac{C}{3b^3}    \E[Z^\rho_{0,(m,n)}]
\;  \le \;    bN^{2/3} +   \tfrac12N^{2/3}  +   \tfrac13    \E[Z^\rho_{0,(m,n)}]. 
\end{align*}  
From this we obtain the bound
$   \E[Z^\rho_{0,(m,n)}]   \le     C_1(\e,\kappa) N^{2/3} $. 
  Substituting this   back into line \eqref{L82} gives the conclusion \eqref{t11} for $\ell\ge\ell_0(\e, \kappa)$.   By increasing the constant $B(\e,\kappa)$ we can cover all $\ell\ge 1$. 
\end{proof}

We state a corollary that quantifies the effect of deviating the endpoint  from the characteristic direction.

\begin{corollary}\label{cor-lb8}    For  $0<\e<\tfrac12$ and $\kappa>0$  
     there exists a finite constant  $C(\e, \kappa)$ such that  for $m,n,N,b\ge 1,  $   
\be\label{lb-23}   \P\bigl\{Z^\rho_{0,(m,\,n+\fl{bN^{2/3}})}\ge 1\bigr\} \le C(\e, \kappa){b^{-3}}
\ee
and 
\be\label{lb-24}   \P\bigl\{Z^\rho_{0,(m,\,n- \fl{bN^{2/3}})} \le  -1\bigr\} \le C(\e, \kappa){b^{-3}}
\ee
whenever these conditions hold: $\rho\in[\e, 1-\e]$,     $\abs{(m,n)-N\xi(\rho)}_1 \le \kappa$, and in the case of \eqref{lb-24} also $n- \fl{bN^{2/3}}\ge 1$.
\end{corollary}

\begin{proof}    For \eqref{lb-23}   introduce another scaling parameter $M$ and a constant $d$  via
\[ M\xi_2(\rho)= n+bN^{2/3} \quad\text{and}\quad d=b\bigl(\tfrac{1-\rho}\rho\bigr)^2 \ge b\e^2. \] 
Then   $\fl{M\xi_2(\rho)}= n+\fl{bN^{2/3}}$ while  
\begin{align*}
M\xi_1(\rho)   = \frac{n(1-\rho)^2}{\rho^2} + d N^{2/3}  = m + d N^{2/3} + \frac{n (1-\rho)^2-m\rho^2}{\rho^2} , 
\end{align*}
from which follows 
\[    \fl{M\xi_1(\rho)} \ge  m + \fl{b\e^2 N^{2/3}}   - \kappa \e^{-2}.  \] 

By  the shifting  Lemma \ref{lm:shift}  in Appendix \ref{app:lpp}, 
\begin{align*}
 &\P\{Z^\rho_{0,(m,\,n+\fl{bN^{2/3}})}\ge 1\} 
 \le  \P\bigl\{Z^\rho_{0,(\fl{M\xi_1(\rho)}\,,  \,\fl{M\xi_2(\rho)}  )}\ge b\e^{2}N^{2/3} -  \kappa \e^{-2} \bigr\}  \\[4pt] 
&\qquad\qquad
 \le  \P\bigl\{Z^\rho_{0,(\fl{M\xi_1(\rho)}\,,  \,\fl{M\xi_2(\rho)}  )}\ge \tfrac12b\e^2N^{2/3}  \bigr\}  
\le {C(\e)}{b^{-3}}. 
\end{align*}
In the second-last inequality we assumed  $b\ge 2\kappa\e^{-4}$  which entails no loss of generality because we can adjust $C(\e,\kappa)$.   The last inequality is from   the upper bound \eqref{t8}. 

\medskip 

For bound \eqref{lb-24} apply again 
  Lemma \ref{lm:shift}  in Appendix \ref{app:lpp} and then   the upper bound \eqref{t8}: 
\[ 
 \P\{Z^\rho_{0,(m,\,n- \fl{bN^{2/3}})} \le  -1\} \le  \P \{Z^\rho_{0,(m,n)} \le  - bN^{2/3}\} \le {C(\e, \kappa)}{b^{-3}}.   
\qedhere \] 
\end{proof}

For directions $\xi=(\xi_1,\xi_2)\in\,]e_2, e_1[$, $x$-coordinates $m\in \Z$,  and $t>0$ define
\begin{align}
	\mathcal{C}_{m,t}^\xi= \{m\}\times \bigl\{ y\in\Z:\abs{m{\xi_2}/{\xi_1}-y}\leq tN^\frac{2}{3}\bigr\}.
\end{align}
$\cC_{m,t}^\xi$ is the  vertical  line segment of length $2tN^{2/3}$ centered on the $\xi$-directed ray  at point   $(m, m{\xi_2}/{\xi_1})$.     Recall that   $\pi^{0,p}$ denotes  the unique  geodesic  of  $G_{0,p}$ that uses i.i.d.~Exp(1) weights. 
The  next lemma shows that for large $r$  the geodesic $\pi^{0,\fl{\xi N}}$ is very  likely to intersect $\cC_{m,t}^\xi$.

\begin{figure}
\includegraphics[width=9.5cm]{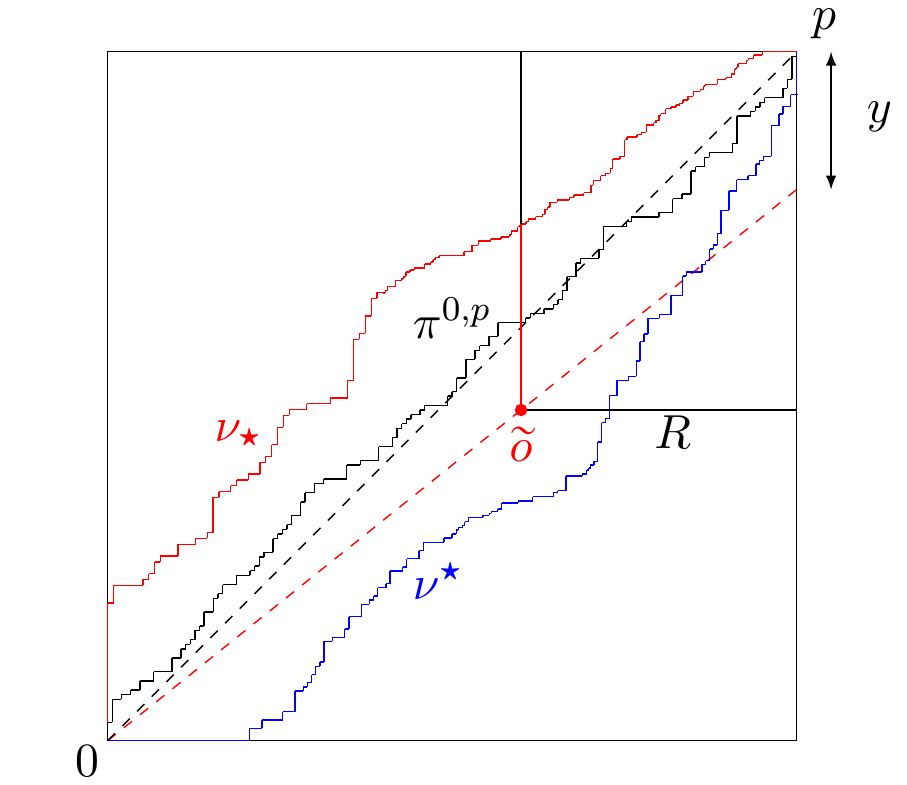} 
\vspace{-10pt} 
\caption{\small Illustration of the proof of Lemma \ref{lm:geodev}.   On the event $\{Z^{\rho_\star}_{0,p}<0, Z^{\rho^\star}_{0,p}>0\}$, geodesic $\nu_\star$ exits off the $y$-axis and $\nu^\star$ off the $x$-axis.   Dashed straight lines:  $[0,p]$ is the ray in direction $\xi$,  $[0,\wt o]$ in direction $\xi_\star$.   With high probability the geodesics 
	$\nu_\star$ and $\nu^\star$ 
	 sandwich the geodesic $\pi^{0,p}$, while not wandering too far from the $\xi$-directed ray.}
\label{fig:geod3} 	 
\end{figure}

\begin{lemma}\label{lm:geodev}
	For  $0<\delta,\e<\frac{1}{3}$, there exists a finite constant $C=C(\delta, \e)$   such that the following holds  for all $N\ge 1$ and $1\le r<  \frac{\sqrt\e}{2(1+\sqrt\e)}N^{1/3}$:   for any direction $\xi=(\xi_1, 1-\xi_1)\in\,]e_2, e_1[$ such that $\xi_1\in\bigl[\frac{\e}{1+\e}, \frac{1}{1+\e}\bigr]$ and any  $i\in \lzb\delta N\xi_1 ,(1-\delta) N\xi_1\rzb$, 
	\be\label{dev76}
	\P\bigl(\pi^{0,\fl{N\xi}}\cap \mathcal{C}^{\xi}_{i,r}= \varnothing\bigr)\leq Cr^{-3}.
	\ee
\end{lemma}

	\begin{proof}  Abbreviate  $p=(p_1,p_2)=\fl{\xi N}$.
	The  proof  shows that  with high probability $\pi^{0,p}$ is   captured  between two geodesics  of stationary LPP processes, and then controls  the probability that these geodesics  deviate from  the  $\xi$-ray.   Figure \ref{fig:geod3} illustrates the proof. 
	
	Take  $\rhoup{\rho}=\rho(\xi)+rN^{-\frac{1}{3}}$ and  $\rhodown{\rho}=\rho(\xi)-rN^{-\frac{1}{3}}$ with characteristic directions  $\xiup{\xi}{}=\xi(\rhoup{\rho})$ and   $\xidown{\xi}{}=\xi(\rhodown{\rho})$.   The upper bound on $r$ guarantees that 
	$\rhoup{\rho}, \rhodown{\rho}\in[\e',1-\e']$.    Let $\geoup{\nu}$ be the geodesic of   $\Gpp^{\rhoup{\rho}}_{0,p}$  and  $\geodown{\nu}$  the geodesic of   $\Gpp^{\rhodown{\rho}}_{0,p}$.  
	We couple the weights of the three LPP processes as follows.  The bulk weights $\{\w_x\}_{x\,\in\, \Z_{>0}^2}$ are the same for each LPP process.  On the axes we couple so that, for $i,j \ge 1$, 
	\be\label{wIJ}
	\w_{ie_1}\le I^{\rhodown{\rho}}_{ie_1} \wedge I^{\rhoup{\rho}}_{ie_1}
	\quad\text{and}\quad 
	\w_{je_2}\le J^{\rhodown{\rho}}_{je_2} \wedge J^{\rhoup{\rho}}_{je_2}.  
	\ee
	
	We develop estimates to control the location of $\geodown{\nu}$. Similar reasoning applies to $\geoup{\nu}$.
	The mean value theorem applied  to the function $\xi_2/\xi_1=(\frac\rho{1-\rho})^2$ shows that  there exist constants  $C_1(\e,\delta),C_2(\e,\delta)>0$ such that
	\begin{align}\label{SC}
			C_1rN^\frac{2}{3}\leq \biggl(\frac{\xi_2}{\xi_1}-\frac{\xidown{\xi}{2}}{\xidown{\xi}{1}}\biggr)i\leq C_2rN^\frac{2}{3}\quad\text{for } \  i\in \lzb\delta N,N\rzb. 
	\end{align}
	 
	Given  $\alpha\in [\delta\xi_1,(1-\delta)\xi_1]$ define the point $\wt o=(\,\fl{\alpha N}\,,\fl{\alpha N{\xidown{\xi}{2}}/{\xidown{\xi}{1}}}\,)$ on the   $\xidown{\xi}{}$-ray.  Let $G^{\rhodown{\rho},[0]}_{\wt o,\,p}$ be the stationary LPP process  on the rectangle $R=\lzb \wt o,p\rzb$ with boundary weights   on the south and west sides given for $i,j\ge 1$  by 
	\[   I^{\rhodown{\rho},[0]}_{\wt o+ie_1}= \Gpp^{\rhodown{\rho}}_{0,\wt o+ie_1}
	- \Gpp^{\rhodown{\rho}}_{0,\wt o+(i-1)e_1}
	\quad\text{and}\quad 
	J^{\rhodown{\rho},[0]}_{\wt o+je_2}= \Gpp^{\rhodown{\rho}}_{0,\wt o+je_2}
	- \Gpp^{\rhodown{\rho}}_{0,\wt o+(j-1)e_2}.
	\]
	Superscript $[0]$ indicates that the boundary weights come from $\Gpp^{\rhodown{\rho}}_{0,\bbullet}$.  
	By Lemma \ref{app-lm1},  the crossing point of the geodesic $\geodown{\nu}$ through the south and west boundary of $R$ is the exit point  
	 of the geodesic of $G^{\rhodown{\rho},[0]}_{\wt o,\,p}$ from that boundary. By \eqref{SC}
	\begin{align*}
		\big\{\geodown{\nu}\cap\mathcal{C}^\xi_{\fl{\alpha N},\,2C_2r}= \varnothing\big\}&\subset \big\{\geodown{\nu} \cap \lzb \wt o,\wt o+2C_2rN^\frac{2}{3}e_2\rzb=\varnothing\big\}\\
		&\subset\big\{Z^{\rhodown{\rho},\,[0]}_{\wt o,p}\notin\lzb-2C_2rN^\frac{2}{3},0\rzb\big\}. 
	\end{align*}
	From this, 
	\be\label{nu in}\begin{aligned}
		\P\bigl(\geodown{\nu}\cap\mathcal{C}^\xi_{\fl{\alpha N},\,2C_2r}= \varnothing\bigr)\leq \P\bigl(Z^{\rhodown{\rho}}_{\wt o,p}\notin\lzb-2C_2rN^\frac{2}{3},0\rzb\bigr)\\
		=\P\bigl(Z^{\rhodown{\rho}}_{\wt o,p}>0\bigr)+\P\bigl(Z^{\rhodown{\rho}}_{\wt o,p}<-2C_2rN^\frac{2}{3}\bigr). 
	\end{aligned}\ee
	(The superscript $[0]$ can be dropped from $Z^{\rhodown{\rho}, [0]}_{\wt o,p}$ in  probability statements because it makes no difference to the distribution.) 
	We show that the last two probabilities are small.  Let
	\begin{align*}
		y=p_2-\frac{\xidown{\xi}{2}}{\xidown{\xi}{1}}N\xi_1   
	\end{align*}
	be the vertical distance between the rays  $\xi$ and $\xidown{\xi}{}$ along the east boundary of $R$. By \eqref{SC}, 
	\begin{align*}
		C_1rN^\frac{2}{3}\leq y\leq C_2rN^\frac{2}{3}.
	\end{align*}
	Since  $p-ye_2-\wt o$ points  in the characteristic direction of   $\rhodown{\rho}$, the bounds below follow from   \eqref{lb-23}  and  \eqref{t8}  for a constant $C=C(\e,\delta)$, uniformly for $\xi_1\in\bigl[\frac{\e}{1+\e}, \frac{1}{1+\e}\bigr]$ and $\alpha\in [\delta\xi_1,(1-\delta)\xi_1]$: 
		\[  \P(Z^{\rhodown{\rho}}_{\wt o,p}>0)\leq Cr^{-3} \]
		and (with first an application of  Lemma \ref{lm:shift}), 
		\begin{align*}
		\P(Z^{\rhodown{\rho}}_{\wt o,p}<-2C_2rN^\frac{2}{3})
		&=\P\bigl(Z^{\rhodown{\rho}}_{\wt o,p-\fl ye_2}<-2C_2rN^\frac{2}{3}+\fl y\,\bigr)\\ &\leq\P(Z^{\rhodown{\rho}}_{\wt o,p-\fl ye_2}<-C_2rN^\frac{2}{3})\leq Cr^{-3}.		
	\end{align*}
	
	Substituting this into \eqref{nu in} gives  a constant $C(\e,\delta)$ independent of $\xi$,  $\alpha$ such that
	\begin{align*}
	\P(\geodown{\nu}\cap\mathcal{C}^\xi_{\fl{\alpha N},\,2C_2r}= \varnothing)\leq Cr^{-3}.
	\end{align*}
	Similarly one shows that
	\begin{align*}
		\P(\geoup{\nu}\cap\mathcal{C}^\xi_{\fl{\alpha N},\,2C_2r}= \varnothing)\leq Cr^{-3}.
	\end{align*}
	Combining the bounds above  with   Corollary \ref{cor-lb8}  gives the next estimate, still with a  constant $C(\e,\delta)$ independent of $\xi, \alpha$:  
	\[  
		\P\bigl\{Z^{\rhodown{\rho}}_{0,p}<0, \, Z^{\rhoup{\rho}}_{0,p}>0,\,\geoup{\nu}\cap\mathcal{C}^\xi_{\fl{\alpha N},\,2C_2r}\neq \varnothing,\,\geodown{\nu}\cap\mathcal{C}^\xi_{\fl{\alpha N},\,2C_2r}\neq \varnothing\bigr\}\geq 1-Cr^{-3}. 
	\]  
	
	The proof of the lemma is complete once we show that the event above implies the complement of \eqref{dev76}, namely, that 
	\be\label{cont} \begin{aligned}
		&\big\{Z^{\rhodown{\rho}}_{0,p}<0, \, Z^{\rhoup{\rho}}_{0,p}>0, \, \geoup{\nu}\cap\mathcal{C}^\xi_{\fl{\alpha N},\,2C_2r}\ne \varnothing, \, \geodown{\nu}\cap\mathcal{C}^\xi_{\fl{\alpha N},\,2C_2r}\ne \varnothing\big\}\\
		&\qquad 
		\subset\big\{\pi^{0,p}\cap \mathcal{C}^{\xi}_{\alpha N,\,2C_2r}\ne  \varnothing\big\}.
	\end{aligned}\ee

The inclusion \eqref{cont}   holds because conditions $Z^{\rhodown{\rho}}_{0,p}<0, \, Z^{\rhoup{\rho}}_{0,p}>0$ imply that the geodesic  $\pi^{0,p}$ runs between geodesics $\geodown{\nu}$ and   $\geoup{\nu}$, with $\geodown{\nu}$ above $\pi^{0,p}$  and   $\geoup{\nu}$ below and to the right of $\pi^{0,p}$.   This is where the coupling \eqref{wIJ} comes in. 
We argue one of the two cases, namely 
\be\label{Z-pi1} \begin{aligned} 
\text{$Z^{\rhoup{\rho}}_{0,p}>0$ implies that $\pi^{0,p}$ never goes strictly  to the right of  $\geoup{\nu}$.} 
\end{aligned} \ee
Let $n=\abs{p}_1$ so that the geodesics end at  $\pi^{0,p}_n=\geoup{\nu}_n=p$. 
Suppose  claim \eqref{Z-pi1} fails, so that at some index $k$,   $\pi^{0,p}_k=\geoup{\nu}_k=z$  but   $\geoup{\nu}_{k+1}=z+e_2$   while $\pi^{0,p}_{k+1}=z+e_1$. 
$Z^{\rhoup{\rho}}_{0,p}>0$ implies that $k\ge 1$ and $z+e_2$ lies in the bulk $\Z_{>0}$.   Since $\pi_{[k+1, n]}$ did not follow the bulk path  $\geoup{\nu}_{[k+1,n]}$,  the bulk weight of $\pi_{[k+1, n]}$ must be  strictly larger than that of $\geoup{\nu}_{[k+1,n]}$.  But now the first inequality of \eqref{wIJ} guarantees that path segment $\geoup{\nu}_{[k+1,n]}$ is inferior to  $\pi_{[k+1, n]}$ also for the stationary LPP value $\Gpp^{\rhoup{\rho}}_{0,p}$. Thus the separation did not happen. 
\end{proof}

\section{No bi-infinite geodesic away from the axes} \label{s:body}

This section proves Theorem \ref{th:ub77}. 
 Recall the southwest boundary part 
 $\partial^N  =(\{-N\}\times \lzb -N,-\e N \rzb\,)\cup(\lzb -N,-\e N\rzb \times \{-N\})$  from \eqref{sw}. 
The parameter $\e>0$  stays fixed now  and hence will be suppressed from some notation.  
As in \eqref{xi-rho}, a point $o=(o_1, o_2)\in \partial^N$ is associated    with its direction  vector $ \xi(o)=(\xi_1(o), 1-\xi_1(o))\in\,]e_2, e_1[$ and rate parameter  $\rho(o)\in (0,1)$   through the relations
\begin{align}
\xi(o)&=  \left(\frac{o_1}{o_1+o_2},\,\frac{o_2}{o_1+o_2}\right)
=\left(\frac{(1-\rho(o))^2}{(1-\rho(o))^2+\rho(o)^2}\,,\frac{\rho(o)^2}{(1-\rho(o))^2+\rho(o)^2}\right)\label{xi-o} \\ 
\intertext{and} 
\rho(o)&=\frac{\sqrt{\abs{o_2}}}{\sqrt{\abs{o_1}}+\sqrt{\abs{o_2}}}
=\frac{\sqrt{1-\xi_1(o)}}{\sqrt{\xi_1(o)}+\sqrt{1-\xi_1(o)}}. 
\label{rho-o}
\end{align}
 For all $o\in \partial^N$ we have the bounds 
\[   \xi(o)\in \Bigl[\Bigl(\frac{\e}{1+\e},\frac{1}{1+\e}\Bigr),\Bigl(\frac{1}{1+\e},\frac{\e}{1+\e}\Bigr)\Bigr] 
\quad\text{and}\quad 
\frac{\sqrt\e}{1+\sqrt\e} \le \rho(o)\le \frac{1}{1+\sqrt\e}. 
\]

The proof uses LPP values from points of $\partial^N$ to the vertical segment 
  $\cI=\{0\}\times\lzb -N^\frac{2}{3},N^\frac{2}{3} \rzb$. This latter is indexed by   $I=\lzb-N^\frac{2}{3},N^\frac{2}{3}\rzb$. For  $o\in \partial^N$, let 
  \[  \rhodown{\rho}(o)=\rho(o)-rN^{-\frac{1}{3}}\quad\text{ and }\quad 
  \rhoup{\rho}(o)=\rho(o)+rN^{-\frac{1}{3}}\]
   and consider the stationary LPP processes  $G^{\rhoup{\rho}}_{o,\bbullet}$ and $G^{\rhodown{\rho}}_{o,\bbullet}$ based at $o$.  
   The  next lemma shows that a  large enough  $r$ forces the exit point of $G^{\rhoup{\rho}}_{o,x}$  to the $x$-axis  and that of $G^{\rhodown{\rho}}_{o,x}$ to the $y$-axis, arbitrarily far on the $N^{2/3}$ scale, with a  probability bound that is uniform  over   $o\in \partial^N$ and $x\in \cI$.    See Figure \ref{fig:geod1} for an illustration.  
   
   \begin{figure} 
 \includegraphics[width=8cm]{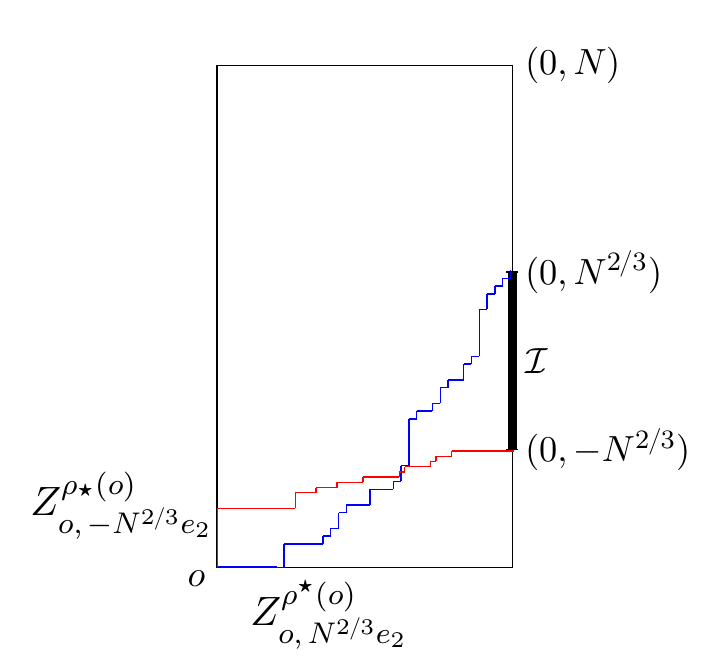} 
 \caption{\small   Lemma \ref{lem-lb1}. For large $r$,   the exit point  $Z^{\rho^\star(o)}_{o,  \fl{N^{2/3}}e_2}$  is far to the right from $o$ on the scale $N^{2/3}$. By the uniqueness of finite geodesics,  the same holds for $Z^{\rho^\star(o)}_{o, x}$ for  all $x\in\cI$. Similarly for exit points $Z^{\rho_\star(o)}_{o,x}$ above $o$.}
 \label{fig:geod1} 
  \end{figure}

	\begin{lemma}\label{lem-lb1}   For $0<\e<1$ there exist  finite positive constants  $C_0(\e), C_1(\e)$  such that, whenever  $d$ and $r$ satisfy 
\be\label{lb-0} 
1\le d\le \tfrac12\e N^{1/3}\quad \text{and}\quad   
C_0(\e)d \le r \le\frac{\sqrt\e}{2(1+\sqrt\e\,)}N^{1/3} ,
\ee
  the following bounds hold for all $N\ge 1$,    $x\in \cI$,  and $o\in \partial^N$: 
		\be\label{lb-1}   \P\bigl\{ Z^{\rhodown{\rho}(o)}_{o,x}\ge -dN^{{2}/{3}}\bigr\} \le  {C_1(\e)}{r^{-3}}
		\ee
		and 
		\be\label{lb-2}   \P\bigl\{ Z^{\rhoup{\rho}(o)}_{o,x}\le  dN^{{2}/{3}}\bigr\}  \le {C_1(\e)}{r^{-3}},
		\ee
where  $\rhodown{\rho}(o)=\rho(o)-rN^{-\frac{1}{3}}$ and $\rhoup{\rho}(o)=\rho(o)+rN^{-\frac{1}{3}}$.
	\end{lemma}

\begin{proof}
	 The upper bound $r \le\frac{\sqrt\e}{2(1+\sqrt\e\,)}N^{1/3}$ guarantees that $\rhodown{\rho}(o), \rhoup{\rho}(o) \in[\e', 1-\e']$ for all $\rho(o)$ and hence the estimates from the increment-stationary CGM apply. 
	
	We   prove \eqref{lb-2}. \eqref{lb-1} is similar.   Represent $o\in\partial^{N\!,\,\e}$ as $o=-(aN, bN)$ where $a\vee b=1$ and $a\wedge b\in[\e , 1]$. 
	Abbreviate $\rho=\rho(o)$ and  $\rhoup{\rho}=\rhoup{\rho}(o)$.   Then   $a/b = (\frac{1-\rho}{\rho})^2$.   
	
	Uniqueness of geodesics forces the $o$ to $\fl{N^{{2}/{3}}}e_2$ geodesic to stay above the $o$ to $x\in\cI$ geodesic. Then apply 
	Lemma  \ref{app-lm1} and translate $o$ to the origin $0$ to deduce: 
	\be \label{706}  
	\begin{aligned}
	 \P\bigl\{ Z^{\rhoup{\rho}}_{o,x}\le dN^{{2}/{3}}\bigr\} & \leq \P\bigl\{ Z^{\rhoup{\rho}}_{o,\,\fl{N^{{2}/{3}}}e_2}\le dN^{{2}/{3}}\bigr\}
	 =\P\bigl\{ Z^{\rhoup{\rho}}_{o, \,\fl{N^{{2}/{3}}}e_2-\fl{dN^{{2}/{3}}}e_1}\le -1\bigr\} \\
	 &= \P\bigl\{ Z^{\rhoup{\rho}}_{0,\, (aN-\fl{dN^{{2}/{3}}},\, bN+ \fl{N^{{2}/{3}}})}\le -1\bigr\}. 
	\end{aligned}\ee
 Define a new scaling parameter $M$ by 
\[   aN-dN^{{2}/{3}} = M\xi_1(\rhoup{\rho}).   \]  
The assumption $d\le \tfrac12\e N^{1/3}    \le \tfrac12a N^{1/3} $  guarantees that $M>0$. 

To apply  \eqref{lb-24}   to the last probability in  \eqref{706}, we bound  the deviation  of   $bN+ \fl{N^{{2}/{3}}}$  from the characteristic point $M\xi_2(\rhoup{\rho})$. 
\begin{align*}
&M\xi_2(\rhoup{\rho}) -  bN-   \fl{N^{{2}/{3}}} 
\ge M\xi_2(\rhoup{\rho}) -  bN-   N^{{2}/{3}} \\[3pt]  
&\qquad = (aN-dN^{{2}/{3}}) \biggl(\frac{\rhoup{\rho}}{1-\rhoup{\rho}}\biggr)^{\!2} 
- \,aN \biggl( \frac{\rho}{1-\rho}\biggr)^{\!2} -  \, N^{{2}/{3}}  \\[3pt] 
&\qquad
= N^{2/3} \biggl(  ar\,\frac{\rhoup{\rho}+\rho-2\rho\rhoup{\rho}}{(1-\rhoup{\rho})^2(1-\rho)^2}   -d  \Bigl(\frac{\rhoup{\rho}}{1-\rhoup{\rho}}\Bigr)^{\!2}  -\, 1 \biggr)  \\[3pt] 
&\qquad
= M^{2/3} \cdot  \frac{\aaa{\rhoup{\rho}}^{2/3}(1-\rhoup{\rho})^{4/3}}{(a-dN^{-1/3})^{2/3}}  \cdot  \frac{ar(\rhoup{\rho}+\rho-2\rho\rhoup{\rho})  -d   (\rhoup{\rho})^2 (1-\rho)^{2} \, -\,  (1-\rhoup{\rho})^2(1-\rho)^2 }{(1-\rhoup{\rho})^2(1-\rho)^2}  . 
\end{align*} 
The above followed from definitions \eqref{xi-rho} and \eqref{a-xi}.  Next bound the last line from below. 
The assumption
$  r \le\frac{\sqrt\e}{2(1+\sqrt\e\,)}N^{1/3} $  guarantees that 
\[     \rhoup{\rho}+\rho-2\rho\rhoup{\rho}  
\ge  c_6(\e)  \] 
for a positive constant $c_6(\e)$ whose precise value is immaterial. 
Use additionally  $\aaa{\rhoup{\rho}}\ge 1$,  $a\ge\e$ and $d\ge 1$ 
to get  the lower bound 
\begin{align*}
M\xi_2(\rhoup{\rho}) -  bN-   \fl{N^{{2}/{3}}}    
\ge  M^{2/3}   (  c_6(\e)\e  r  -2d   ) 
\ge 
M^{2/3} c_7(\e)   r  
\end{align*} 
where the last inequality follows from   assuming 
$    r\ge  4d c_6(\e)^{-1}\e^{-1} \equiv C_0(\e)d$ and  defining  $c_7(\e)$ suitably.    
Returning to \eqref{706}, we have 
\begin{align*}
\P\bigl\{ Z^{\rhoup{\rho}}_{o,x}\le dN^{{2}/{3}}\bigr\} 
	 &= \P\bigl\{ Z^{\rhoup{\rho}}_{0,\, (aN-\fl{dN^{{2}/{3}}},\, bN+ \fl{N^{{2}/{3}}})}\le -1\bigr\}\\
	 &\le  \P\bigl\{ Z^{\rhoup{\rho}}_{0,\, ( \,\fl{ M\xi_1(\rhoup{\rho})}  ,\, \fl{M\xi_2(\rhoup{\rho})}  - \fl{M^{2/3}    r \,c_7(\e)}\, )}\le -1\bigr\}    \le  C_1(\e) r^{-3}.  
\end{align*}
The last inequality comes from \eqref{lb-24}.   The constant $\kappa$ in \eqref{lb-24}    can be fixed at 2 and ignored. 
\end{proof}

We introduce a pair of parameters   $d=(d_1,d_2)\in \Z_{\ge 1}^2$ that control coarse graining on the scale $N^{2/3}$, $d_1$  on the southwest portion of the boundary of  the square $\lzb-N,N\rzb^2$ and $d_2$ on the northeast part.    For  $o\in \partial^N$  let 
 \begin{align}
 	\cI_{o,d}=\{u \in \partial^N:\abs{u-o}_1\leq \tfrac12{d_1}N^\frac{2}{3}\}  
 \end{align}
 and 
  \be\label{o_c}
  \text{$o_c=$ the unique minimal point of $\cI_{o,d}$ in the coordinatewise partial order on $\Z^2$. }
  \ee
 For an illustration of $o$, $o_c$ and $\cI_{o,d}$  see Figure \ref{fig:points}.  
 
 For a given point $o\in\partial^N$, we compare the LPP processes $G_{u,\bbullet}$ from initial points $u\in\cI_{o,d}$ with increment-stationary LPP processes $G^{\rhodown{\rho}}_{o_c,\bbullet}$  and $G^{\rhoup{\rho}}_{o_c,\bbullet}$ with base point $o_c$ and parameters 
 \[  \rhodown{\rho}=\rho(o_c)-rN^{-\frac{1}{3}} \quad\text{ and }\quad \rhoup{\rho}=\rho(o_c)+rN^{-\frac{1}{3}}, \] 
assumed to satisfy  $\rhodown{\rho}, \rhoup{\rho}\in(0,1)$.    The weights on the boundaries with corner at $o_c$ are coupled as in \eqref{wIJ}:  
 for $i,j \ge 1$, 
	\be\label{wIJ14}
	\w_{o_c+ie_1}\le I^{\rhodown{\rho}}_{o_c+ie_1} \wedge I^{\rhoup{\rho}}_{o_c+ie_1}
	\quad\text{and}\quad 
	\w_{o_c+je_2}\le J^{\rhodown{\rho}}_{o_c+je_2} \wedge J^{\rhoup{\rho}}_{o_c+je_2}.  
	\ee

Associated to these LPP processes are vertical increment variables  on the  $y$-axis.  We are concerned now only on the range $j\in I$, so the increment variables below are well-defined once  $-\e N < -N^{2/3}$.    For $u\in\cI_{o,d}$ and   $\rho\in\{\rhodown{\rho}, \rhoup{\rho}\}$, let 
\begin{align*}
	J^u_j=G_{u,je_2}-G_{u,(j-1)e_2} 
	 \quad\text{and}\quad 
J^{\rho}_j=G^{\rho}_{o_c, je_2}-G^{\rho}_{o_c,(j-1)e_2}, \quad  j\in I.  
\end{align*}
	Define  the event
	\be\label{Aod} 
		A_{o,d} = \left\{Z^{\rhodown{\rho}}_{o_c,-\ce{N^{{2}/{3}}}e_2}<-d_1N^{\frac{2}{3}},Z^{\rhoup{\rho}}_{o_c, \ce{N^{{2}/{3}}}e_2}>d_1N^{\frac{2}{3}}\right\}.  
	\ee 

\begin{lemma}\label{lem ge1}   Let $N\ge N_0(\e)$ so that the increment variables are well-defined for $j\in I$. 
	On the event $A_{o,d}$ we have the inequalities 
	\be\label{Aod1} 
		J^{\rhoup{\rho}}_j\leq J^u_j \leq J^{\rhodown{\rho}}_j\quad \forall j\in I,u\in \cI_{o,d}.
	\ee  
	There exists a constant $C(\e)$ such that, whenever
	 $(d_1, r)$ satisfy \eqref{lb-0}, then 
	 	\be\label{Aod2}   \P\bigl(A_{o,d}^c\bigr) \le C_1(\e) r^{-3} \quad \text{ for all $o\in\partial^N$.}    \ee 
\end{lemma}
\begin{proof}   We prove the second inequality of \eqref{Aod1}. The first one  comes analogously.   

Let $\wt G_{x,y}$ be the LPP process on the quadrant $o_c+\Z_{\ge0}^2$  that uses weights $\wt\w$ defined by $\wt\w_{o_c+je_2}=J^{\rhodown{\rho}}_{o_c+je_2}$ for $j\ge 1$,  $\wt\w_{o_c}=0$,  and $\wt\w_{o_c+x}=\w_{o_c+x}$ for $x\cdot e_1>0$. 

Suppose first that $u=o_c+\ell e_2$ for some $\ell\ge 0$.   The uniqueness of finite geodesics together with the first inequality of the event $A_{o,d}$ implies  that  $Z^{\rhodown{\rho}}_{o_c, x}<-d_1N^{\frac{2}{3}}$ for all $x\in\cI$.  Hence both   $u$ and $u+e_2$  lie on the geodesic of $G^{\rhodown{\rho}}_{o_c, x}$ for all $x\in\cI$. Consequently 
\[    G^{\rhodown{\rho}}_{o_c, x+e_2}- G^{\rhodown{\rho}}_{o_c, x}
=  \wt G_{u, x+e_2} - \wt G_{u, x}.   \]
   Lemma \ref{lm:G13} gives the inequality 
\[ \wt G_{u, x+e_2} - \wt G_{u, x}  \ge G_{u, x+e_2} - G_{u, x} . \]

The other case is that $u=o_c+k e_1$ for some $k\ge 1$. Then  $G^{\rhodown{\rho}}_{o_c, x} =   \wt G_{o_c, x}$  and $G_{u, x}=\wt G_{u, x}$.  This time the conclusion follows from Lemma \ref{lm:crs}.   


  Bound \eqref{Aod2} comes from Lemma \ref{lem-lb1}. 
\end{proof}
Next we perform  the analogous construction in the northeast quadrant. As in \eqref{ne}, 
$\rim{\partial}^N=\{N\}\times \lzb \e N,N \rzb \cup \lzb \e N,N\rzb \times \{N\}$. 
A point  $\rim{o}=(\rim{o}_1,\rim{o}_2)\in \rim{\partial}^N$ is  associated with a density $\rho(\rim{o})\in (0,1)$ and a direction $ \xi(\rim{o})\in\,]e_2, e_1[$   through the relations  \eqref{xi-o}--\eqref{rho-o}. 
For $x,y\in \Z^2$ such that $x\leq y$   define a reversed last-passage process  $\rim{G}_{y,x}=G_{x,y}$ in terms of the i.i.d.\ Exp(1) $\w$-weights.  

For each parameter value $\rho\in(0,1)$, analogously with \eqref{Gr11a}--\eqref{Gr12a},  we define  stationary last-passage percolation processes   $\rim{G}^\rho_{\rim{o}, \bbullet}$ on the southwest quadrant $\rim{o}+\mathbb{Z}^2_{\le0}$.   	 Let
 $\{\wh I^\rho_{\rim{o}-ie_1}\}_{i\in\Z_{>0}}$
	  and $\{\wh J^\rho_{\rim{o}-je_2}\}_{j\in\Z_{>0}} $  
   be mutually independent  boundary weights   on the north and east,  with marginal distributions $\wh I^\rho_{\rim{o}-ie_1}\sim\text{Exp}(1-\rho)$ and $\wh J^\rho_{\rim{o}-je_2}\sim\text{Exp}(\rho)$, independent of  the boundary variables $I^\rho_{o+ie_1}, J^\rho_{o+je_2}$ in the southwest quadrant. 
     Put  $\rim{G}^\rho_{\rim{o},\,\rim{o}}=0$ and on the boundaries 
\be\label{Gr11} \rim{G}^\rho_{\rim{o},\,\rim{o}-ke_1}=\sum_{i=1}^k \wh I^\rho_{\rim{o}-ie_1} 
\quad\text{and}\quad
\rim{G}^\rho_{\rim{o},\,\rim{o}-le_2}= \sum_{j=1}^l  \wh J^\rho_{\rim{o}-je_2} .   \ee 
Then in the bulk 
for $x=(x_1,x_2)\in \rim{o}+ \Z_{<0}^2$, 
\be\label{Gr12}
\rim{G}^\rho_{\rim{o},\,x}= \max_{1\le k\le \rim{o}_1-x_1} \;  \Bigl\{  \;\sum_{i=1}^k \wh I^\rho_{\rim{o}-ie_1}  + \rim{G}_{ \,\rim{o}-ke_1-e_2,\,x} \Bigr\}  
\bigvee
\max_{1\le \ell\le \rim{o}_2-x_2}\; \Bigl\{  \;\sum_{j=1}^\ell  \wh J^\rho_{\rim{o}-je_2}  + \rim{G}_{ \,\rim{o}-\ell e_2-e_1,\,x} \Bigr\} .
\ee
For a southwest endpoint  $p\in \rim{o}+\Z_{<0}^2$, let $\rim{Z}^\rho_{\rim{o},p}$  be the signed exit point of the geodesic $\rim{\pi}^{\rho,\rim{o},p}_{\bbullet}$ of $\rim{G}^\rho_{\rim{o},p}$ from the north and east boundaries  of $\rim{o}+\Z_{\le0}^2$.   Precisely,
\begin{align}
\rim{Z}^\rho_{\rim{o},\,x}=
\begin{cases}
\argmax{k\ge 1} \bigl\{ \,\sum_{i=1}^k \wh I^\rho_{\rim{o}-ie_1}  + \rim{G}_{\,\rim{o}-ke_1-e_2,\,x} \bigr\},  &\text{if }\rim{\pi}^{\rho,\rim{o},\,x}_1=\rim{o}-e_1,\\
-\argmax{\ell\ge 1}\bigl\{  \;\sum_{j=1}^\ell  \wh J^\rho_{\rim{o}-je_2}  + \rim{G}_{\,\rim{o}-\ell e_2-e_1,\,x} \bigr\},  &\text{if } \rim{\pi}_1^{\rho,\rim{o},\,x}=\rim{o}-e_2.
\end{cases}
\end{align}
 For  $\rim{o}\in \rim{\partial}^N$ 
 let
 \begin{align*}
 \rim{\cI}_{\rim{o},d}=\bigl\{v \in \rim{\partial}^N:\abs{v-\rim{o}}_1\leq \tfrac12{d_2}N^{{2}/{3}}\bigr\} 
 \end{align*}
 and (with a illustration in Figure \ref{fig:points}),  
  \be\label{rimo_c}
  \text{$\rim{o}_c=$ the unique maximal point of $\rim{\cI}_{\rim{o},d}$ in the coordinatewise partial order on $\Z^2$. }
  \ee

Define again parameters 
 \[  \rhodown{\rho}=\rho(\rim{o}_c)-rN^{-\frac{1}{3}} \quad\text{ and }\quad 
 \rhoup{\rho}=\rho(\rim{o}_c)+rN^{-\frac{1}{3}}. \]
Define increment variables    on the  vertical edges $\{ (x+e_1, x+e_1+e_2):  x\in\cI  \}$ shifted by $e_1$ from $\cI$.  For  $v\in \rim{\cI}_{\rim{o},d}$,  $j\in I$,  and $\rho\in\{\rhodown{\rho}, \rhoup{\rho}\}$, let 
\begin{align*}
	\rim{J}^v_j=\rim{G}_{v,e_1+(j-1)e_2}-\rim{G}_{v,e_1+je_2}
	 \quad\text{and}\quad 
\rim{J}^{\rho}_j&=\rim{G}^{\rho}_{\rim{o}_c,\,e_1+(j-1)e_2}-\rim{G}^{\rho}_{\rim{o}_c,\,e_1+je_2}.  
\end{align*}
Define  the event
	\be\label{Bod} 
	B_{\rim{o},d}=\Bigl\{\rim{Z}^{\rhodown{\rho}}_{\rim{o}_c,N^{{2}/{3}}e_2+e_1}<-d_2N^{\frac{2}{3}}, \; \rim{Z}^{\rhoup{\rho}}_{\rim{o}_c,-N^{{2}/{3}}e_2+e_1}>d_2N^{\frac{2}{3}}\Bigr\}.  
	\ee 
 We have this analogue of  Lemma \ref{lem ge1}. 

%

\begin{lemma}\label{lem ge2}    Let $N\ge N_0(\e)$ so that the increment variables are well-defined for $j\in I$. 
	On the event $B_{\rim{o},d}$ we have the inequalities 
\be\label{Bod1} 
	\rim{J}^{\rhoup{\rho}}_j\leq \rim{J}^v_j \leq \rim{J}^{\rhodown{\rho}}_j\, \qquad \forall j\in I,v\in \rim{\cI}_{\rim{o},d}.
	\ee	
	There exists a constant $C(\e)$ such that, whenever
	 $(d_2, r)$ satisfy \eqref{lb-0}, then 
	 	\be\label{Bod2}   \P\bigl(B_{\rim{o},d}^c\bigr) \le C_1(\e) r^{-3} \quad \text{ for all $\rim{o}\in\rim{\partial}^N$.}    \ee 
\end{lemma}

Let $o\in \partial^N,\rim{o}\in\rim{\partial}^N$ and consider LPP from points $u\in \cI_{o,d}$ to the interval $\cI$ on the $y$-axis and reverse LPP from   points $v\in \rim{\cI}_{\rim{o},d}$ to the shifted interval $e_1+\cI$. Abbreviate   $\rhoup{\lambda}=\rhoup{\rho}(\rim{o}_c)$, $\rhodown{\lambda}=\rhodown{\rho}(\rim{o}_c)$ and $\rhoup{\rho}=\rhoup{\rho}(o_c)$ , $\rhodown{\rho}=\rhodown{\rho}(o_c)$. 

 A given sequence of  steps $\{X_j\}$ defines a two-sided walk $S(X)$   by 
\begin{align*}
S_n(X)=
\begin{cases} 
\sum_{j=1}^{n}X_j & n \geq 1 \\
0 & n=0\\
-\sum_{j=n+1}^{0}X_j & n < 0. 
\end{cases}
\end{align*}
Use this notation to define three random walks indexed by the edges 
$\{ ((0,j),(1,j)):  j\in I \}$ that run along the $y$-axis.   The steps are defined 
   by  
\begin{align*}
	X^{u,v}_j=J^u_j-\rim{J}^v_j , \quad 
	Y'_j=J^{\rhodown{\rho}}_j-\rim{J}^{\rhoup{\lambda}}_j, \quad\text{and}\quad  
	Y_j=J^{\rhoup{\rho}}_j-\rim{J}^{\rhodown{\lambda}}_j. 
\end{align*}
The corresponding  walks are denoted by 
\[ 
	S^{u,v}=S(X^{u,v}), \quad 
	S^{'}=S(Y{'}),  \quad\text{and}\quad  
	S=S(Y).  
\]   
Recall the events defined in \eqref{Aod} and \eqref{Bod}. 
\begin{lemma}\label{lm:sw}
	The  processes 
	\be \label{cor b-rw1}
	\{S^{'}_m :  m\in \lzb-N^{2/3} ,-1\rzb\,\}  \quad\text{and}\quad 
	 \{S_n: n\in \lzb1,N^{2/3} \rzb\,\} 
		\quad\text{are independent.} 
	\ee 
	On the event $A_{o,d}\cap B_{\rim{o},d}$, for all $u\in \cI_{o,d}$ and $v\in \rim{\cI}_{\rim{o},d}$,  
	\be\label{sw5} \begin{aligned}
		S_n &\leq S_n^{u,v}\leq S^{'}_n \quad\text{for } \  n\in \lzb1,N^{{2}/{3}}\rzb\\
	\text{and}\quad 
		S^{'}_n &\leq S_n^{u,v}\leq S_n \quad\text{for } \   n\in \lzb-N^{{2}/{3}},-1\rzb.  
	\end{aligned}\ee
\end{lemma}

\begin{proof}  The independence of the stationary LPP processes defined on the southwest and northeast quadrants implies that the processes $\{J^{\rhodown{\rho}}_j,  J^{\rhoup{\rho}}_j\}_{j\in I} $ and  $\{ \rim{J}^{\rhodown{\lambda}}_j, \rim{J}^{\rhoup{\lambda}}_j\}_{j\in I}$ are independent of each other.  Theorem \ref{th:st-lpp}(i)  implies that within these processes,  $\{J^{\rhodown{\rho}}_j\}_{j\le 0}$ and   $\{J^{\rhoup{\rho}}_j\}_{j\ge 1} $ are independent, as are  $\{\rim{J}^{\rhodown{\lambda}}_j\}_{j\ge 1}$ and   $\{\rim{J}^{\rhoup{\lambda}}_j\}_{j\le 0} $.   (Note the switch in the direction of indexing: since the geodesics of $\rim{G}^{\rim{o}_c}_{\rim{o}_c, \bbullet}$  proceed southwest instead of northeast, application of Theorem \ref{th:st-lpp} requires reversal of lattice directions.) 
 
 Inequalities \eqref{sw5} come from the inequalities \eqref{Aod1} and \eqref{Bod1}. 
	\end{proof}
 
 Next observe that  the walk $S^{u,v}(X)$ controls the edge along   which  the geodesic $\pi^{u,v}$ steps away from  the $y$-axis. 
\be\label{eq:var1}\begin{aligned} 
G_{u, v} &= \sup_{u_2 \leq n \leq v_2}\big\{G_{u,(0,n)}+\rim{G}_{v,(1,n)}\big\}\\ \nonumber
&=\sup_{u_2 \leq n \leq v_2}\bigl\{G_{u,(0,0)}+[G_{u,(0,n)}-G_{u,(0,0)}]  
+\rim{G}_{v,e_1}-[\rim{G}_{v, (1,0)}-\rim{G}_{v,(1,n)}]\bigr\}\\ \nonumber
&=\sup_{u_2 \leq n \leq v_2}\{G_{u,(0,0)}+\rim{G}_{v,(1,0)}+S^{u,v}_n\}.
\end{aligned}\ee
In consequence,  
\be\label{S-pi}\begin{aligned}
&\text{for $u\in \cI_{o,d}$ and $v\in \rim{\cI}_{\rim{o},d}$, the geodesic $\pi^{u,v}$ goes along  the edge ${((0,j),(1,j))}$}\\ 
&\text{if and only if  }  j=\argmax{u_2\leq n\le v_2}\{S^{u,v}_n\},    
\end{aligned}\ee
that is, if and only if the almost surely  unique maximum of $S^{u,v}_n$ is taken at $n=j$. 

Let $o\in \partial^N$ and $\rim{o}=-o\in\rim{\partial}^N$  as in Figure \ref{fig:points}.
Let $o_c\in\cI_{o,d}$ be defined by  \eqref{o_c} and $\rim{o}_c\in \rim{\cI}_{\rim{o},d}$ by \eqref{rimo_c}. 
We will take $d_1\ne d_2$, so $\rim{o}_c\ne-o_c$.   
   For $u\in \cI_{o,d}$ and $v\in \rim{\cI}_{\rim{o},d}$ define the event
\begin{align}
	\label{geo-zero}
	&\eventcB^{u,v}=\{\text{geodesic $\pi^{u,v}$ uses  edge $((0,0), (1,0))$}\} . 
\end{align}

\begin{lemma}\label{lm:clo}
	Let $r=N^\frac{2}{15}$ and $d=(1,N^\frac{1}{8})$. 
	There exist constants  $C(\e), N_0(\e)$ such that for all 
	  $o\in \partial^N$  and $N\ge N_0(\e)$, 
	\begin{align}
		\P\biggl(\; \bigcup_{u\,\in\,\cI_{o,d}, \,v\,\in\,\rim{\cI}_{\rim{o},d}}\eventcB^{u,v} \biggr)  
		\leq C(\e)N^{-2/5}.
	\end{align}
\end{lemma}
\begin{proof}
	Fix $o\in \partial^N$. For any $u\in \cI_{o,d},v\in \rim{\cI}_{\rim{o},d}$, by \eqref{S-pi}, 
	\begin{align*}
		\eventcB^{u,v}\subseteq\Big\{\sup_{0 < l\leq N^{{2}/{3}}} S^{u,v}_l<0\Big\}\cap \Big\{\sup_{-N^{{2}/{3}}\leq l<0} S^{u,v}_l<0\Big\}.  
	\end{align*}
	By Lemma \ref{lm:sw}, on the event $A_{o,d}\cap B_{\rim{o},d}$,  
	\begin{align*}
		\bigcup_{u\,\in\,\cI_{o,d}, \,v\,\in\,\rim{\cI}_{\rim{o},d}}\Big\{\sup_{0 < l\leq N^{{2}/{3}}} S^{u,v}_l<0\Big\}&\subseteq\Big\{\sup_{0 < l\leq N^{{2}/{3}}} S_l<0\Big\}\\
	\text{and}\quad 
		\bigcup_{u\,\in\,\cI_{o,d}, \,v\,\in\,\rim{\cI}_{\rim{o},d}}\Big\{\sup_{-N^{{2}/{3}}\le l<0 } S^{u,v}_l<0\Big\}&\subseteq\Big\{\sup_{-N^{{2}/{3}}\le  l<0 } S^{'}_l<0\Big\}. 
	\end{align*}
	Thus on the event $A_{o,d}\cap B_{\rim{o},d}$,  
		\begin{align*}
		\bigcup_{u\,\in\,\cI_{o,d}, \,v\,\in\,\rim{\cI}_{\rim{o},d}}\eventcB^{u,v}\subseteq \Big\{\sup_{0 < l\leq N^{{2}/{3}}} S_l<0\Big\}\cap \Big\{\sup_{-N^{{2}/{3}}\leq l<0} S^{'}_l<0\Big\}.  
	\end{align*}
	By the independence claim of  Lemma \ref{lm:sw}, 
	\begin{align}\label{rw ub}
		\P\Big(\bigcup_{u\,\in\,\cI_{o,d}, \,v\,\in\,\rim{\cI}_{\rim{o},d}}\eventcB^{u,v}\Big) \leq \P\Big(\sup_{0 < l\leq N^{{2}/{3}}} S_l<0\Big) \P\Big(\sup_{-N^{{2}/{3}}\leq l<0} S^{'}_l<0\Big)+\P\big(A_{o,d}^c\cup B_{\rim{o},d}^c\big).
	\end{align}
	
	Let $\rho=\rho(o_c)$ and $\lambda=\rho(\rim{o}_c)$. Since $\rho(\rim{o})=\rho(o)$,  there is a constant $C(\e)$ such that, for $N\ge 1$, 
	\be\label{681}  \abs{\rho-\lambda} \leq C(\e)(d_2+d_1)N^{-\frac{1}{3}}\le C(\e) N^{-5/24}. 
	\ee 
	Each step of the random walk $S$ on $\lzb 1,N^{\frac{2}{3}}\rzb$ is the difference of independent exponential random variables with parameters $\rhoup{\rho}= \rho+rN^{-\frac{1}{3}}$ and $\rhodown{\lambda}= \lambda-rN^{-\frac{1}{3}}$. Similarly, each step of the random walk $S^{'}$ on $\lzb-N^{{2}/{3}},-1\rzb $ is the difference of independent exponential random variables with parameters   $\rhodown{\rho}= \rho-rN^{-\frac{1}{3}}$ and  $\rhoup{\lambda}= \lambda+rN^{-\frac{1}{3}}$. 
Take $r=N^\frac{2}{15}$.  Then for   $N\ge N_0(\e)$, we have $\rhoup{\rho}>\rhodown{\lambda}$.    (By \eqref{681} we can take  $N_0(\e)=C(\e)^{120}$.) 
Inequality   \eqref{rw34.9} with $\alpha=\rhoup{\rho}$ and $\beta=\rhodown{\lambda}$ gives the bound 
\be\label{683} \begin{aligned}
	\P\bigg(\sup_{0 < l\leq N^{{2}/{3}}} S_l<0\bigg)&\leq \frac{C}{N^\frac{1}{3}}\bigg(1-\biggl(\frac{\rho-\lambda+2rN^{-\frac{1}{3}}}{\rho+\lambda}\biggr)^2\;\bigg)^{N^{{2}/{3}}}+\; \frac{\rho-\lambda+2rN^{-\frac{1}{3}}}{\rho+rN^{-\frac{1}{3}}}\\
	&\leq \frac{C}{N^\frac{1}{3}}\bigg(1-\biggl(\frac{\rho-\lambda+2rN^{-\frac{1}{3}}}{\rho+\lambda}\biggr)^2\;\bigg)^{N^{{2}/{3}}}+\; C(d_1+d_2+r)N^{-\frac{1}{3}}.
\end{aligned}\ee
With  $r=N^\frac{2}{15}$ and $d=(d_1, d_2)=(1,N^\frac{1}{8})$, the last line is dominated by the last term.  Thus  there is a constant $C_3(\e)>0$ not depending on $o$, such that 
\begin{align}\label{rw ub2}
	\P\Big(\sup_{0 < l\leq N^{{2}/{3}}} S_l<0\Big) \leq C_3N^{-\frac{1}{5}}.
\end{align}
Similarly one shows that
\begin{align}\label{rw ub3}
	\P\Big(\sup_{-N^{{2}/{3}}\leq l<0} S^{'}_l<0\Big)\leq C_3N^{-\frac{1}{5}}.
\end{align}
With  $r=N^\frac{2}{15}$,  \eqref{Aod2} and \eqref{Bod2} give for $N\ge N_0(\e)$ 
\begin{align}\label{ub 3}
	\P\big(A_{o,d}^c\cup B_{\rim{o},d}^c\big)\leq  Cr^{-3}   = CN^{-\frac{2}{5}}.
\end{align}
To complete the proof,  
substitute \eqref{rw ub2}, \eqref{rw ub3} and \eqref{ub 3} into \eqref{rw ub}.
\end{proof}

\begin{remark}\label{rm:loss}  In the proof above we can observe where the optimal estimate is lost.   Namely, if the probability $\P\big(A_{o,d}^c\cup B_{\rim{o},d}^c\big)$ could be ignored in \eqref{rw ub}, we could take $r$ and $d_2$ to be constants. This would result in the bound $C_3N^{-1/3}$ in \eqref{rw ub2} and \eqref{rw ub3}.  The end result would be an upper bound of order $N^{-2/3}$ on the probability that two opposite blocks of size $N^{2/3}$ are connected by a geodesic through the origin.   Since geodesics fluctuate on the scale $N^{2/3}$, this is the expected order. 
\end{remark} 

  Let $o\in \partial^N$ and   $\rim{o}=-o$ be as before above Lemma \ref{lm:clo}.  For $d=(d_1, d_2)$ set 
\begin{align}
\rim{\cF}_{\rim{o},d}=\{v\in \rim{\partial}^N: \abs{\rim{o}-v}_1>\tfrac12{d_2}N^{\frac{2}{3}}\}
\end{align}

 \begin{figure}
	\begin{subfigure}{.5\textwidth}
	\begin{tikzpicture}[scale=0.5, every node/.style={transform shape}]
	\draw  (0,0) rectangle (10,10);
	\draw [dashed][line width=0.01cm] (0,5) -- (10,5);
	\draw [dashed][line width=0.01cm] (5,0) -- (5,10);
	\foreach \x in {0,...,6}
	{   \draw [fill] (0,\x*1/5) circle [radius=0.07];
	};
	\foreach \x in {0,...,3}
	{   \draw [fill] (\x*1/5,0) circle [radius=0.07];
	};
	\node [scale=2][above] at (1.2,0.3) {$\cI_{o,d}$};
	\node [scale=2][above][red] at (-0.5,-0.7) {$o_c$};
	\node [scale=2][above][cyan] at (-0.5,0.3) {$o$};
	\draw [fill] (0,3/5) circle [radius=0.07][cyan];
	\draw [fill] (0,0) circle [radius=0.07][red];
	\draw [dashed][line width=0.01cm] (0,1/5) -- (0,3/5);
	\foreach \x in {0,...,8}
	{   \draw [fill] (10,10-\x*1/5) circle [radius=0.07];
	};
	\foreach \x in {0,...,8}
	{   \draw [fill] (10-\x*1/5,10) circle [radius=0.07];
	};
	\draw [fill] (10,10) circle [radius=0.07][red];
	\draw [fill] (10,9.4) circle [radius=0.07][cyan];
	\node [scale=2][above][red] at (10,10) {$\rim{o}_c$};
	\node [scale=2][above][cyan] at (10.5,9) {$\rim{o}$};
	\node [scale=2][above] at (8.9,8) {$\rim{\cI}_{o,d}$};
	\node [scale=2][above][blue] at (6.8,8.2) {$\rim{\cF}^1_{\rim{o},d}$};
	\node [scale=2][above][blue] at (8.8,5.6){$\rim{\cF}^2_{\rim{o},d}$}; 
	
	\foreach \x in {9,...,23}
	{   \draw [fill] (10-\x*1/5,10) circle [radius=0.07][blue];
	};
	\foreach \x in {9,...,23}
	{   \draw [fill] (10,10-\x*1/5) circle [radius=0.07][blue];
	};
	\end{tikzpicture}
\end{subfigure}%
	\begin{subfigure}{.5\textwidth}
	\begin{tikzpicture}[scale=0.5, every node/.style={transform shape}]
	\draw  (0,0) rectangle (10,10);
	\draw [dashed][line width=0.01cm] (0,5) -- (10,5);
	\draw [dashed][line width=0.01cm] (5,0) -- (5,10);
	\foreach \x in {1,...,9}
	{   \draw [fill] (0,2+\x*1/5) circle [radius=0.07];
	};
	\node [scale=2][above] at (1.2,2) {$\cI_{o,d}$};
	\node [scale=2][above][red] at (-0.6,1.7) {$o_c$};
	\node [scale=2][above][cyan] at (-0.6,2.6) {$o$};
	\draw [fill] (0,3) circle [radius=0.07][cyan];
	\draw [fill] (0,2.2) circle [radius=0.07][red];
	\draw [dashed][line width=0.01cm] (0,1/5) -- (0,3/5);
	\foreach \x in {1,...,9}
	{   \draw [fill] (10,6+\x*1/5) circle [radius=0.07];
	};
	\draw [fill] (10,7.8) circle [radius=0.07][red];
	\draw [fill] (10,7) circle [radius=0.07][cyan];
	\node [scale=2][above][red] at (10.6,7.3) {$\rim{o}_c$};
	\node [scale=2][above][cyan] at (10.6,6.5) {$\rim{o}$};
	\node [scale=2][above] at (9,6.5) {$\rim{\cI}_{o,d}$};
	\node [scale=2][above][blue] at (7.2,8.2) {$\rim{\cF}^1_{\rim{o},d}$};
	\node [scale=2][above][blue] at (8.8,4.8) {$\rim{\cF}^2_{\rim{o},d}$};
	\foreach \x in {0,...,23}
	{   \draw [fill] (10-\x*1/5,10) circle [radius=0.07][blue];
	};
	\foreach \x in {1,...,4}
	{   \draw [fill] (10,5.2+\x*1/5) circle [radius=0.07][blue];
	};
	\foreach \x in {1,...,10}
	{   \draw [fill] (10,10-\x*1/5) circle [radius=0.07][blue];
	};		
	\end{tikzpicture}
\end{subfigure}
\caption{\small The square  $\lzb-N,N\rzb^2$ with  two possible arrangements of the segments ${\cI}_{o,d}$, $\rim{\cI}_{o,d}$ and $\rim{\cF}_{\rim{o},d}=\rim{\cF}^1_{\rim{o},d}\cup\rim{\cF}^2_{\rim{o},d}$ on the boundary of the square. In both cases $\rim{o}=-o$.}  
\label{fig:points}
  \end{figure}

\begin{lemma} \label{lm:far}
	Let $d=(1,N^\frac{1}{8})$. There are finite constants $C(\e)$ and $N_0(\e)$ such that, for any   $N\ge N_0(\e)$ and  $o\in \partial^N$,  
	\be\label{far7} 
			\P\Big(	\bigcup_{u\,\in\,\cI_{o,d},\,v\,\in\,\rim{\cF}_{\rim{o},d}} \eventcB^{u,v}\Big)\leq C(\e)N^{-\frac{3}{8}}
	\ee
\end{lemma}
\begin{proof}
	Define the   sets of boundary points
	\begin{align}
		&\partial \rim{\cF}_{\rim{o},d}=\{v\in \rim{\cF}_{\rim{o},d}:\exists u\in \rim{\cI}_{\rim{o},d} \text{ such that }  v\sim u\}\\
		&\partial \cI_{o,d}=\{v\in \cI_{o,d} :\exists u\in \partial^N\setminus \cI_{o,d}  \text{ such that }  v\sim u\},
	\end{align}
	where   $v\sim u$ means  that $v$ and $u$ are adjacent in the graph $\Z^2$.  Their cardinalities satisfy $1\le |\partial \rim{\cF}_{\rim{o},d}|\leq   |\partial \cI_{o,d}| \leq 2$.  (For example,  $\partial \rim{\cF}_{\rim{o},d}$ is a singleton if $\rim{\cI}_{\rim{o},d}$ contains one of the endpoints $(N, \fl{\e N})$ or $(\fl{\e N}, N)$  of $\rim{\partial}^N$.) 
We denote  the points of $\partial \rim{\cF}_{\rim{o},d}$ by $q^1,q^2$  and those of $\partial \cI_{o,d}$  by $h^1,h^2$, 
labeled so that these inequalities are satisfied:
\begin{align*}
		&q^1_1\leq \rim{o}_1\le q^2_1, \quad q^1_2\geq \rim{o}_2\ge  q^2_2\\
		&h^1_1\geq o_1 \ge h^2_1, \quad h^1_2\leq o_2 \le  h^2_2. 
	\end{align*}
Geometrically,  starting from the north pole $(0,N)$ and  traversing the boundary of the square $\lzb-N,N\rzb^2$ clockwise, we meet the points (those that exist) in this order:  $q^1\to \rim{o}\to q^2\to h^1\to o\to h^2$ (Figure \ref{fig:bigdev}). 

	
	
	For points $u\in \partial^N,v\in \rim{\partial}^N$  let 
	\begin{align*}
		\mathcal{P}_m^{u,v}=\pi^{u,v}\cap \big\{x\in \Z^2:x_1=m\big\} 
	\end{align*}
	be the   intersection of the geodesic $\pi^{u,v}$ with the vertical line at $x_1=m$.
	For $t>0$  let 
	\be\label{Duv} 
		\eventcD^{u,v}_{m,t}=\Big\{\inf_{p=(p_1, \,p_2)\,\in\, \mathcal{P}_m^{u,v}}\Big| u_2+ \frac{v_2-u_2}{v_1-u_1}(m-u_1)-p_2\Big|>t\Big\}.
	\ee
	be  the event that along this vertical line    the geodesic $\pi^{u,v}$ deviates by distance  at least  $t$ from the straight line segment from $u$ to $v$.   We now show that the event in \eqref{far7} implies that one of the geodesics $\pi^{\,h^i\!, \,q^i}$  deviates by at least order $d_2N^{2/3}$ from the straight line segment $[h^i, q^i]$. 
	
	
	For $o\in \partial^N$,  $u\in \partial \cI_{o,d}$ and $v\in\partial \rim{\cF}_{\rim{o},d} $  decompose as  $u=o+e^u$ and $v=\rim{o}+e^v$.  These vectors satisfy  
	\be\label{e^uv}   \abs{e^u}_1\leq \tfrac12d_1N^\frac{2}{3}, \quad  \abs{e^v}_1\geq \tfrac12d_2N^\frac{2}{3}, \quad  \abs{e^v_1}\vee \abs{e^v_2}\leq 2(1-\e)N, \quad\text{and}\quad e^v_1e^v_2\leq 0. \ee  
	
	   $\rim{\cF}_{\rim{o},d}$ is the union of two disjoint pieces separated by $\rim{\cI}_{\rim{o},d}$, one of which can be empty.   $\rim{\cF}_{\rim{o},d}^1$ is to the left and above $\rim{\cI}_{\rim{o},d}$ separated from $\rim{\cI}_{\rim{o},d}$ by the point $q^1$.  $\rim{\cF}_{\rim{o},d}^2$ is  to the right and below $\rim{\cI}_{\rim{o},d}$ separated from $\rim{\cI}_{\rim{o},d}$ by the point $q^2$. They  can be expressed  as follows: 
	\begin{align*}
	\rim{\cF}_{\rim{o},d}^1=\{ v\in \rim{\cF}_{\rim{o},d}:  e^v_1\le 0\le e^v_2\}  
	\quad\text{and}\qquad \rim{\cF}_{\rim{o},d}^2=\{ v\in \rim{\cF}_{\rim{o},d}:  e^v_2\le 0\le e^v_1\}. 
	\end{align*} 
	
	Decompose the point appearing  in \eqref{Duv} suitably, 
	  using  $v_i-u_i= \rim{o}_i+e^v_i-(o_i+e_i^u) = -2o_i +e^v_i- e_i^u$. 
	\begin{align}
		 u_2+ \frac{v_2-u_2}{v_1-u_1} (m-u_1) 
		 &=o_2- \frac{v_2-u_2}{v_1-u_1}o_1 +e^u_2  + \frac{v_2-u_2}{v_1-u_1} (m-e^u_1) \nn \\[4pt] 
	\label{234} 	&= \frac{o_2e^v_1-o_1e^v_2}{v_1-u_1}  - \frac{o_2e^u_1-o_1e^u_2}{v_1-u_1} +e^u_2  + \frac{v_2-u_2}{v_1-u_1} (m-e^u_1) . 
		 \end{align} 	
		 The first term on the last line is of order $\Theta(d_2N^{2/3})$ because  there is no cancellation in the numerator.    It   is positive if $v\in\rim{\cF}_{\rim{o},d}^1$ and negative if $v\in\rim{\cF}_{\rim{o},d}^2$.   This term   dominates because $d_2>>d_1$.

 From the calculation above we bound  signed vertical distances from  the  $x$-axis to the line segment $[u,v]$.  In addition to \eqref{e^uv}, we utilize $\;-N\le o_i\le -\e N$,  $2N\e\le v_i-u_i\le 2N$ and   the slope bound $\e\le \frac{v_2-u_2}{v_1-u_1}\le \e^{-1}$. 
 
First for $u\in\cI_{o,d}$ and  $v\in \rim{\cF}_{\rim{o},d}^1$ we bound below the positive distance from the origin to $[u,v]$ so we take $m=0$. The $e^u$-terms on line \eqref{234} are collected together into a single error term. 
\be\label{237-1} \begin{aligned}
		 u_2+ \frac{v_2-u_2}{v_1-u_1}(-u_1) 
		 &\ge  \frac{\e N  \abs{e^v}_1}{ 2N}    - \Bigl(   \frac{N}{2N\e} +1+\e^{-1}\Bigr) \abs{e^u}_1  \\[4pt] 
		 &\ge    \tfrac14\e d_2N^\frac{2}{3}   - 2\e^{-1}  d_1N^\frac{2}{3}  
		 \ge  \tfrac1{8}\e d_2N^\frac{2}{3}. 
\end{aligned} \ee
In the last inequality we used $(d_1, d_2)=(1,N^{1/8})$ and took $N\ge (16\e^{-2})^8$.

For  $u\in\cI_{o,d}$ and  $v\in \rim{\cF}_{\rim{o},d}^2$ we  bound above the negative  distance from the point $(1,0)$  to $[u,v]$ and hence take $m=1$: 
\be\label{237-2} \begin{aligned}
		 u_2+ \frac{v_2-u_2}{v_1-u_1}(1-u_1)  
		 &\le - \,\frac{\e N  \abs{e^v}_1}{ 2N}    + \Bigl(   \frac{N}{2N\e} +1+\e^{-1}\Bigr) \abs{e^u}_1 +\e^{-1} \\[4pt] 
		 &
		 \le   - \tfrac14\e d_2N^\frac{2}{3}   + 3\e^{-1}  d_1N^\frac{2}{3}\le - \tfrac1{8}\e d_2N^\frac{2}{3} . 
\end{aligned}\ee

Now suppose that for some $u\in\cI_{o,d}$ and  $v\in \rim{\cF}_{\rim{o},d}$ the geodesic $\pi^{u,v}$ goes through the edge $((0,0),(1,0))$.   We have two cases. 
\begin{enumerate} [{(i)}]  \itemsep=4pt

\item 
If $v\in \rim{\cF}_{\rim{o},d}^1$,  then the  geodesic $\pi^{\,h^1\!,\,q^1}$ stays below and to the right of $\pi^{u,v}$ because both its endpoints are below and to the right of the endpoints of $\pi^{u,v}$.  Then \eqref{237-1} with $u=h^1$ and $v=q^1$  shows that at $x$-coordinate $x=0$ the  geodesic $\pi^{\,h^1\!,\,q^1}$ deviates from the straight line segment $[h^1, q^1]$ by at least $\tfrac1{8}\e d_2N^\frac{2}{3}$.   This case is illustrated in Figure \ref{fig:bigdev}. 

\item If $v\in \rim{\cF}_{\rim{o},d}^2$,  then the  geodesic $\pi^{\,h^2\!,\,q^2}$ stays above  and to the left of $\pi^{u,v}$.  Now \eqref{237-2} with $u=h^2$ and $v=q^2$  shows that at $x$-coordinate $x=1$ the  geodesic $\pi^{\,h^2\!,\,q^2}$ deviates from the straight line segment $[h^2, q^2]$ by at least $\tfrac1{8}\e d_2N^\frac{2}{3}$. 
\end{enumerate}

\vspace{3pt} 

	Put  cases (i) and (ii)  together and  apply Lemma \ref{lm:geodev}: 
		\begin{align}
		\P\bigg(\;\bigcup_{u\,\in\,\cI_{o,d},\,v\,\in\,\rim{\cF}_{\rim{o},d}} \eventcB^{u,v}\bigg) \le 
		\P\bigl(  \eventcD^{h^1\!,\,q^1}_{0,\, C_1(\e)d_2N^{{2}/{3}}}   \cup \eventcD^{h^2\!,\,q^2}_{1,\, C_1(\e)d_2N^{{2}/{3}}}   \bigr) 
		\leq   C(\e)d_2^{-3} = C(\e)N^{-\frac{3}{8}}. 
	\end{align}
	The proof is complete.
\end{proof}

\begin{figure}
\includegraphics[width=8.3cm]{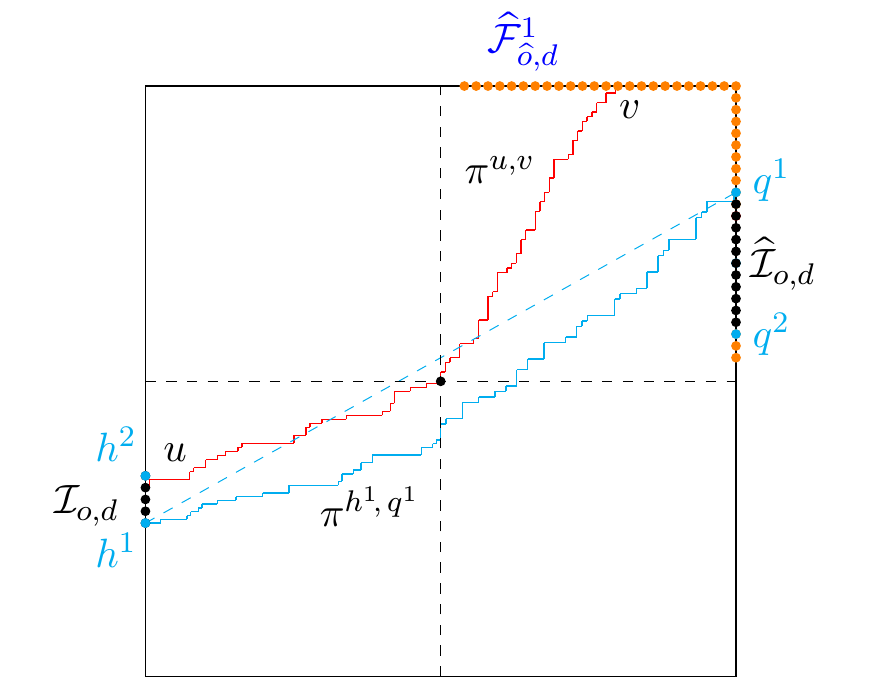} 
\caption{\small Case (i) in the proof of Lemma \ref{lm:far}.  The  geodesic $\pi^{u,v}$ connects  $\cI_{o,d}$ and   $\rim{\cF}^1_{\rim{o},d}$   through the edge $((0,0),(1,0))$.    The  geodesic $\pi^{\,h^1\!, \,q^1}$  lies below $\pi^{u,v}$ and hence well below the $[h^1, q^1]$ line segment (dashed line).}
\label{fig:bigdev}
\end{figure}

\begin{lemma}\label{lm:ub43}
 	There is a constant $C(\e)$ such that for any  $o\in \partial^N$,  
 	\begin{align}
 	\P\Big(	\bigcup_{u\,\in\,\cI_{o,d}, \,v\,\in\,\rim{\partial}^N} \eventcB^{u,v}\Big)\leq C(\e)N^{-\frac{3}{8}}
 	\end{align}
\end{lemma}
\begin{proof}
     Since $\rim{\partial}^N=\rim{\cI}_{\rim{o},d}\cup \rim{\cF}_{\rim{o},d}$, 
     \begin{align}
     \P\Big(	\bigcup_{u\,\in\,\cI_{o,d}, \,v\,\in\,\rim{\partial}^N} \eventcB^{u,v}\Big)\leq \P\Big(\bigcup_{u\,\in\,\cI_{o,d},\,v\,\in\,\rim{\cF}_{\rim{o},d}} \eventcB^{u,v}\Big)+\P\Big(	\bigcup_{u\,\in\,\cI_{o,d}, \,v\,\in\,\rim{\cI}_{\rim{o},d}}\eventcB^{u,v}\Big)
     \end{align}
     and  Lemmas \ref{lm:clo} and   \ref{lm:far} give the claimed bound. 
\end{proof}

We come to the final step of the proof that   geodesics that connect $\partial^N$ and $\rim{\partial}^N$ through the origin are rare.  
 Recall the event $W_{N\!,\,\e}$ defined in \eqref{WNe}. 

\begin{proof}[Proof of Theorem \ref{th:ub77}] 
A geodesic through the origin takes after that either an $e_1$ or an $e_2$ step.  By symmetry it suffices to control only one case.  We prove 
\be\label{ub43} 
			\P\Big(	\bigcup_{u\,\in\,\partial^N,\,v\,\in\,\rim{\partial}^N} \eventcB^{u,v}\Big)\leq C(\e)N^{-\frac{1}{24}}
	\ee 
for the event  $\eventcB^{u,v}$ 	defined in \eqref{geo-zero}. 
As before,  $d=(1,N^\frac{1}{8})$.   To coarse graine  $\partial^N$ let 
 \begin{align*}
	\mathcal{O}^N=  \partial^N \cap \Big(\big\{(-N +id_1\fl{N^{\frac{2}{3}}}\,,-N)\big\}_{i\in\Z_{\ge0}} \; \bigcup  \,\big\{(-N,-N+jd_1\fl{N^{\frac{2}{3}}})\big\}_{j\in\Z_{\ge0}}\Big).
\end{align*}
	Then decompose 
	\begin{align*}
			\bigcup_{u\,\in\,\partial^N,\,v\,\in\,\rim{\partial}^N} \eventcB^{u,v} \subseteq \bigcup_{o\,\in\, \mathcal{O}^N} 	\bigcup_{u\,\in\,\cI_{o,d},\,v\,\in\, \rim{\partial}^N} \eventcB^{u,v}.
	\end{align*}
	As $|\mathcal{O}^N|\leq C(\e) d_1^{-1}N^{1-\frac{2}{3}}=C(\e)N^\frac{1}{3}$, a  union bound and Lemma \ref{lm:ub43} give \eqref{ub43}: 
	\[ 
	 	\P\Big(\bigcup_{u\,\in\,\partial^N,\,v\,\in\,\rim{\partial}^N} \eventcB^{u,v}\Big) \leq \sum_{o\,\in\,\mathcal{O}^N} \P\Big(	\bigcup_{u\,\in\,\cI_{o,d}, \,v\,\in\,\rim{\partial}^N} \eventcB^{u,v}\Big)\leq C(\e)N^{\frac{1}{3}}N^{-\frac{3}{8}}=C(\e)N^{-\frac{1}{24}}.
		\qedhere 
	\] 
\end{proof}
 
 \medskip 
 
\section{No nontrivial axis-directed  geodesic} 
\label{s:no-axis}

First we complete the proof  of  Theorem \ref{th:no-axis} with the lemma below and then prove the lemma. 

\begin{lemma}\label{u:lm7} 
Let $\eta_k=(\eta_{k,1}, 1-\eta_{k,1})\in\,]e_2, e_1[$ be a monotone sequence of directions such that $\eta_{1,1} <\eta_{2,1}<\dotsm<\eta_{k,1}<\dotsm$ and  $\lim_{k\to\infty} \eta_k= e_1$.   Let  $w_{n,k}=w(n,k) = (\fl{n\eta_{k,1}}, n-\fl{n\eta_{k,1}})\in\Z_{>0}^2$ be lattice points such that $\lim_{n\to\infty}n^{-1} w_{n,k}=\eta_k$ for each $k$.  Then 
\be\label{u:67} 
\varlimsup_{k\to\infty} \; \varlimsup_{n\to\infty} \, \bigl[  G_{0,w(n,k)} -G_{e_2, w(n,k)} \bigr] =\infty
\qquad\text{$\P$-almost surely.} 
\ee
\end{lemma}

\smallskip 

\begin{proof}[Proof of Theorem \ref{th:no-axis}]     It is enough to  prove the case $\evec_1$ for $x=0$.   

Fix $\eta_k$ and $w_{n,k}$ as in Lemma \ref{u:lm7} and let $\Omega_0$ be the event of full probability on which \eqref{u:67} holds.  Fix $\w\in\Omega_0$ and 
suppose that  at this $\w$  there is a   semi-infinite geodesic $\pi=\{\pi_n\}_{n\in\Z_{\ge0}}$  such that $\pi_0=0$, $\pi_\ell=(\ell-1,1)$ for some $\ell\ge 1$, and $\varliminf_{n\to\infty} n^{-1} \pi_n\cdot \evec_2=0$.  We derive a contradiction.   

By connecting $\evec_2=(0,1)$ to the point $\pi_\ell=(\ell-1,1)$ (now fixed)  with a horizontal path, we get the lower bound 
\[  \Gpp_{\evec_2, \pi_n}\ge  \sum_{i=0}^{\ell-1}\w_{(i,1)} + \Gpp_{\pi_{\ell+1}, \pi_n} 
\qquad \text{for } n>\ell. \]
That $\pi$ is a geodesic from $\pi_0=0$ implies  $\Gpp_{0, \pi_n}= \Gpp_{0, \pi_\ell} + \Gpp_{\pi_{\ell+1}, \pi_n} $ for $n>\ell$. 
Thus 
\be\label{406}  \Gpp_{0, \pi_n} -  \Gpp_{\evec_2, \pi_n}\le \Gpp_{0, \pi_\ell}-  \sum_{i=0}^{\ell-1}\w_{(i,1)}
 \qquad \text{for all } n>\ell. \ee
 By the assumptions $\varliminf n^{-1} \pi_n\cdot \evec_2=0$ and $\eta_k\in\,]e_2, e_1[$, and  by the crossing lemma,  for each $k$  there are infinitely many indices $n$ such that 
 \[ \Gpp_{0, \pi_n} -  \Gpp_{\evec_2, \pi_n}\ge \Gpp_{0, w_{n,k}} -  \Gpp_{\evec_2, w_{n,k}}. 
  \]  
Hence  for each $k$, 
 \[  \varlimsup_{n\to\infty} [\Gpp_{0, \pi_n} -  \Gpp_{\evec_2, \pi_n} ] \ge \varlimsup_{n\to\infty} [ \Gpp_{0, w_{n,k}} -  \Gpp_{\evec_2, w_{n,k}}] . 
  \]  
Limit \eqref{u:67} now contradicts \eqref{406} because the right-hand side of \eqref{406} is fixed and finite. 
\end{proof} 

\begin{proof}[Proof of Lemma \ref{u:lm7}] 
Let $r<\infty$ and begin by bounding as follows: 
\be\label{u:102} \begin{aligned} 
&\P\bigl\{ \; \varlimsup_{k\to\infty} \,  \varlimsup_{n\to\infty} \, \bigl[  G_{0,w(n,k)} -G_{e_2, w(n,k)} \bigr]  \ge  r\bigr\}  
\ge \; \varlimsup_{k\to\infty} \, \P\bigl\{ \;  \varlimsup_{n\to\infty} \, \bigl[  G_{0,w(n,k)} -G_{e_2, w(n,k)} \bigr]  >  r\bigr\}.  
\end{aligned} \ee
We show that the last probability converges to one as $k\to\infty$.  

Choose parameters $\lambda_k$ so that 
\be\label{u:112}  1 > \lambda_k> \rho(\eta_k)=\frac{\sqrt{1-\eta_{k,1}}}{\sqrt{1-\eta_{k,1}}+\sqrt{\eta_{k,1}}}.
\ee
Define the reverse stationary LPP processes $\wh G^{\lambda_k}_{w(n,k), x}$ for $x\in w_{n,k}+\Z_{<0}^2$ as in \eqref{Gr11}--\eqref{Gr12},  with parameter $\lambda_k$ and   northeast base point  $w_{n,k}$. 
As before,  for $x\in w_{n,k}+\Z_{<0}^2$,  let 
\[    \wh J^{\lambda_k}_{w(n,k), x} = \wh G^{\lambda_k}_{w(n,k), x} - \wh G^{\lambda_k}_{w(n,k), x+\evec_2} \]
denote vertical increment variables with  distribution 
$\wh J^{\lambda_k}_{w(n,k), x}\sim{\rm Exp}(\lambda_k)$. 
Similarly to the argument in Lemma \ref{lem ge2}, when the geodesic of $\wh G_{w(n,k),0}$  takes a $-e_1$ step from $w_{n,k}$, that is,  $ \wh Z^{\lambda_k}_{w(n,k), 0} >0$,
 the  increments satisfy 
\be\label{u:115} 
\wh J^{\lambda_k}_{w(n,k), 0} \le  
\wh G_{w(n,k),0} - \wh G_{w(n,k), \evec_2} = G_{0,w(n,k)} -G_{e_2, w(n,k)}.
\ee
The inequality  follows from a combination of Lemmas   \ref{lm:G13} and \ref{lm:crs}. 

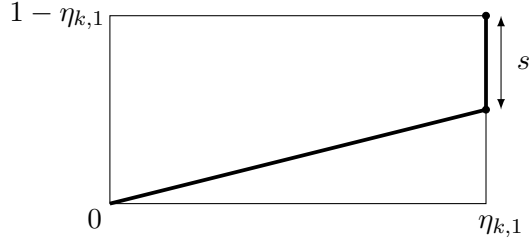
\begin{figure}[pt]
	\begin{tikzpicture}[scale=0.5, every node/.style={transform shape}]
	\draw  (0,0) rectangle (10,5);
	\draw [line width=0.05cm] (0,0) -- (10,2.5);
	\draw [line width=0.05cm] (10,2.5) -- (10,5);
	\draw [fill] (10,2.5) circle [radius=0.09];
	\draw [fill] (10,5) circle [radius=0.09];
	\draw[<->,>=latex,black]   (10.4,2.5) --(10.4,5);
	\node [scale=2]at (-0.4,-0.4) {$0$};
	\node [scale=2]at (-1.4,5) {$1-\eta_{k,1}$};
	\node [scale=2]at (10.4,-0.6) {$\eta_{k,1}$};
	\node [scale=2]at (11,3.75) {$s$};
	\end{tikzpicture}
	\caption{\small When the geodesic  is forced to go  downward from the northeast corner, the geodesic chooses   the   distance $s$  on the east  side  to maximize the sum of   Exp$(\lambda_k)$ weights on the east side and  the bulk LPP value between the origin and the point $(\eta_{k, 1}, 1-\eta_{k,1}-s)$.}
	\label{fig:timo5} 
\end{figure}

To take advantage of this we
record the limiting shape functions.  The stationary LPP process satisfies almost surely 
\be\label{u:104} 
\lim_{n\to\infty}  n^{-1}  \wh G^{\lambda_k}_{w(n,k), 0}
= \frac{\eta_{k,1}}{1-\lambda_k} +  \frac{1-\eta_{k,1}}{\lambda_k} . 
\ee
Let $\wh G^{\lambda_k}_{w(n,k), 0}\bigr[\rim{Z}^{\lambda_k}_{w(n,k), 0} <0\bigr]$ denote the last-passage value computed by maximizing over only those paths that satisfy the condition  $\rim{Z}^{\lambda_k}_{w(n,k), 0} <0$, or equivalently, that take first a $-e_2$ step from $w_{n,k}$.  The limit can be  calculated  from  a macroscopic variational formula (see Figure \ref{fig:timo5} for justification):  
\be\label{u:107} \begin{aligned} 
\lim_{n\to\infty}  
n^{-1}  \wh G^{\lambda_k}_{w(n,k), 0}\bigr[\rim{Z}^{\lambda_k}_{w(n,k), 0} <0\bigr]
&=\sup_{0\le s\le 1-\eta_{k,1}} \Bigl\{  \frac{s}{\lambda_k} + \gpp(\eta_{k,1}, 1-\eta_{k,1}-s)\Bigr\}  \\
& 
=  \gpp(\eta_{k,1}, 1-\eta_{k,1}).  
\end{aligned} \ee
That the supremum is achieved at $s=0$  is a consequence of \eqref{u:112}.    Increasing $\lambda_k$ strictly above the characteristic value $\rho(\eta_k)$ as in \eqref{u:112} has the effect that the geodesic of  $\wh G^{\lambda_k}_{w(n,k), 0}$  spends  a macroscopic distance on the horizontal boundary $w_{n,k}+\Z_{<0}e_1$. 
 Hence forcing the $-e_2$ step from the corner $w_{n,k}$ is suboptimal, and  it can be checked directly that 
\be\label{u:117}
 \frac{\eta_{k,1}}{1-\lambda_k} +  \frac{1-\eta_{k,1}}{\lambda_k} 
 >   \gpp(\eta_{k,1}, 1-\eta_{k,1}).    
 \ee

We deduce a probability bound from \eqref{u:115}. 
\begin{align*}
&\P\bigl(  G_{0,w(n,k)} -G_{e_2, w(n,k)} \le r\bigr) 
\le \P\bigl( \wh Z^{\lambda_k}_{w(n,k), 0} <0\bigr)  + \P( \wh J^{\lambda_k}_{w(n,k), 0} \le r)\\[4pt]  
&\qquad
=   \P\bigl\{ \, \wh G^{\lambda_k}_{w(n,k), 0}  =  \wh G^{\lambda_k}_{w(n,k), 0}[\rim{Z}^{\lambda_k}_{w(n,k), 0} <0]  \, \bigr\} 
+  1 -e^{-\lambda_kr} .  
\end{align*} 
   By \eqref{u:104}, \eqref{u:107},  and \eqref{u:117},  the first probability  on the last line vanishes as $n\to\infty$.   Switch to complements to get 
\[  
 \varliminf_{n\to\infty}  \P\bigl(  G_{0,w(n,k)} -G_{e_2, w(n,k)} > r\bigr) \ge  e^{-\lambda_kr} .
 \] 
From this,  upon replacing $r$ by $r+\e$  for  $\e>0$, 
\begin{align*}
&\P\bigl\{ \;  \varlimsup_{n\to\infty} \, \bigl[  G_{0,w(n,k)} -G_{e_2, w(n,k)} \bigr]  >  r\bigr\} \\
&\quad \ge 
\P\bigl\{  G_{0,w(n,k)} -G_{e_2, w(n,k)}  > r+\e \text{ for infinitely many $n$} \bigr\}
\ge   e^{-\lambda_kr-\lambda_k\e} .  
\end{align*}  
 
  By assumption $\eta_{k,1}\to 1$.  Hence we can satisfy \eqref{u:112} while also having $\lambda_k\to 0$.    Thus the  lower bound   in \eqref{u:102} equals  one.
\end{proof}

\bigskip

\appendix

\section{Queues} \label{app:queues}

 We formulate  last-passage percolation over a bi-infinite strip as  a   queueing operator.   
 The inputs are   two bi-infinite sequences: the {\it inter-arrival process}   $\arrv=(\arr_j)_{j\in\Z}$ and the {\it service process} $\servv=(\serv_j)_{j\in\Z}$.  
 The queueing  interpretation is that $\arr_j$ is the time between the arrivals of customers $j-1$ and $j$ and    $\serv_j$ is the service time of customer $j$.  The operations below are well-defined as long as ${\lim_{m\to-\infty}}  \sum_{i=m}^0 (\serv_i-\arr_{i+1})=-\infty.$  
 
   From  inputs $(\arrv,\servv)$  three output sequences 
\be\label{DSR}   \depav=\Dop(\arrv,\servv),  \quad \sojov=\Sop(\arrv,\servv), \quad\text{and} \quad \wc \servv=\Rop(\arrv,\servv) \ee
  are constructed through explicit mappings: the {\it inter-departure process} $\depav=(\depa_j)_{j\in\Z}$, the {\it sojourn process} $\sojov=(\sojo_j)_{j\in\Z}$, and the {\it dual service times}  $\wc\servv=(\wc\serv_j)_{j\in\Z}$.  
  
  The  formulas are as follows.  Choose a sequence  $G=(G_j)_{j\in\Z}$  that satisfies $\arr_j=G_j-G_{j-1}$.   Define the sequence  $\wt G=(\wt G_j)_{j\in\Z}$ by 
\be\label{m:800}
\wt G_j=\sup_{k:\,k\le j}  \Bigl\{  G_k+\sum_{i=k}^j \serv_i\Bigr\}. 
 \ee
The supremum above  is taken  at some finite $k$.   Then set 
\be\label{DSR4} 
\depa_j =  \wt G_j - \wt G_{j-1}, \quad 
\sojo_j=\wt G_j -  G_j, \quad\text{and} \quad 
\wc \serv_j=\arr_j\wedge \sojo_{j-1}. \ee
The outputs \eqref{DSR4} do not  depend on the choice of $G$ as  long as $\arr_j=G_j-G_{j-1}$. Note that to compute $\{ \depa_j , \sojo_j, \wc \serv_j: j\le m\}$,  only inputs $\{ \arr_j , \serv_j : j\le m\}$ are needed.

  The next lemma is a  deterministic property of the mappings. 
  
\begin{lemma}\label{lm:DR}   The identity 
$\Dop\bigl(\Dop(\barrv, \arrv), \servv\bigr) = \Dop\bigl( \Dop(\barrv, \Rop(\arrv, \servv)), \Dop(\arrv, \servv)\bigr)$ holds whenever the sequences $\arrv, \barrv, \servv$ are such that the operations are well-defined.  
\end{lemma} 

\begin{proof}
Choose $(A_j)$ and $(B_j)$ so that $A_j-A_{j-1}=\arr_j$ and $B_j-B_{j-1}=\barr_j$.  Then the output of $\Dop(\barrv, \arrv)$  is  the increment sequence of  
\[  \wt B_\ell=\sup_{k\le \ell} \Bigl\{  B_k +  \sum_{i=k}^\ell \arr_i\Bigr\} . \]  
Next,  the output of 
 $\Dop(\Dop(\barrv, \arrv), \servv)$ is  the increment sequence of 
\[   H_m= \sup_{\ell\le m} \Bigl\{  \wt B_\ell +  \sum_{j=\ell}^m \serv_j\Bigr\}
= \sup_{k\le m} \Bigl\{  B_k +  \max_{\ell:\, k\le\ell\le m}  \Bigl[ \, \sum_{i=k}^\ell \arr_i +  \sum_{j=\ell}^m \serv_j\Bigr] \Bigr\}. \]

Similarly,  define first 
\[  \wt A_j=\sup_{k:\,k\le j}  \Bigl\{  A_k+\sum_{i=k}^j \serv_i\Bigr\}
\quad\text{and}\quad 
\wc B_\ell=\sup_{k\le \ell} \Bigl\{  B_k +  \sum_{i=k}^\ell \wc\serv_i\Bigr\}.  
\]
Then  the output of  $\Dop\bigl( \Dop(\barrv, \Rop(\arrv, \servv)), \Dop(\arrv, \servv)\bigr)$ is  the increment sequence of 
\[   \wt H_m= \sup_{\ell\le m} \Bigl\{  \wc B_\ell +  \sum_{j=\ell}^m \wt\arr_j\Bigr\}
= \sup_{k\le m} \Bigl\{  B_k +  \max_{\ell:\, k\le\ell\le m}  \Bigl[ \, \sum_{i=k}^\ell \wc\serv_i +  \sum_{j=\ell}^m \wt\arr_j\Bigr] \Bigr\}. \]

It remains to check that 
\be\label{DR7} 
 \max_{\ell:\, k\le\ell\le m}  \Bigl[ \, \sum_{i=k}^\ell \wc\serv_i +  \sum_{j=\ell}^m \wt\arr_j\Bigr]
 = \max_{\ell:\, k\le\ell\le m}  \Bigl[ \, \sum_{i=k}^\ell \arr_i +  \sum_{j=\ell}^m \serv_j\Bigr] . 
\ee
This can be verified with a case-by-case analysis. See Lemma 4.3 in \cite{fan-sepp-arxiv}. 
\end{proof} 

Specialize to stationary  M/M/1 queues.   Let $\sigma$ be a service rate and $\alpha_1, \alpha_2$ arrival rates.  Assume 
  $\sigma>\alpha_1>\alpha_2>0$.  Let $\barrv^1, \barrv^2, \servv$ be   mutually independent i.i.d.\ sequences with marginals   $\barr^k_j\sim\text{Exp}(\alpha_k)$ for $k\in\{1,2\}$   and $\serv_j\sim\text{ Exp}(\sigma)$.    Define a jointly distributed pair of arrival sequences by 
  $(\arrv^1, \arrv^2)= \bigl( \barrv^1, \Dop( \barrv^2, \barrv^1)\bigr)$.  From these and   services $\servv$, define jointly distributed output variables: 
  \[  
   \depav^k=\Dop(\arrv^k,\servv),  \quad \sojov^k=\Sop(\arrv^k,\servv), \quad\text{and} \quad \wc \servv^k=\Rop(\arrv^k,\servv)
   \quad \text{  for } k\in\{1,2\}.  
   \]  

\begin{lemma}\label{lm:DR5}   We have the following properties. 
\begin{enumerate}[{\rm(i)}]   \itemsep=3pt 
\item Marginally $\arrv^2$ is a sequence of i.i.d.\ ${\rm Exp}(\alpha_2)$ variables. 
\item For   fixed   $k\in\{1,2\}$ and $m\in\Z$, the random variables   $\{\depa^k_j\}_{j\le m}$, $\sojo^k_m$, and $\{\wc\serv^k_j\}_{j\le m}$ are mutually independent with marginal distributions $\depa^k_j\sim\text{\rm Exp}(\alpha_k)$, $\sojo^k_m\sim\text{\rm Exp}(\sigma-\alpha_k)$, and $\wc\serv^k_j\sim\text{\rm Exp}(\sigma)$. 
\item For  a fixed   $k\in\{1,2\}$, sequences $\depav^k$ and $\wc \servv^k$ are mutually independent sequences of i.i.d.\ random variables with marginal distributions $\depa^k_j\sim\text{\rm Exp}(\alpha_k)$ and $\wc\serv^k_j\sim\text{\rm Exp}(\sigma)$. 

\item  $(\depav^1, \depav^2)\deq(\arrv^1, \arrv^2)$, in other words, we have found a distributional fixed point for this joint queueing operator. 

\item   For any $m\in\Z$, the random variables $\{ \arr^2_i\}_{i\le m}$ and $\{\arr^1_j\}_{j\ge m+1}$ are mutually independent.  


\end{enumerate} 
\end{lemma} 

\begin{proof}
Parts (i)--(iii) are  basic M/M/1 queueing theory.   Proofs can be found for example in Lemma B.2 in Appendix B of \cite{fan-sepp-arxiv}.  

For part  (iv),   the marginal distributions of $\depav^1$ and $\depav^2$ are the correct ones  by Lemma \ref{lm:DR5}(iii). To establish the correct joint distribution,   the definition of $(\arrv^1, \arrv^2)$ points us to find an i.i.d.\ Exp$(\alpha_2)$  random sequence $\zvec$ that is independent of $\depav^1$ and satisfies $\depav^2=   \Dop( \zvec, \depav^1)$.  From the definitions and  Lemma \ref{lm:DR}, 
\begin{align*}
\depav^2=\Dop(\arrv^2,\servv) =  \Dop\bigl( \Dop( \barrv^2, \arrv^1),\servv\bigr)
= \Dop\bigl( \Dop(\barrv^2, \Rop(\arrv^1, \servv)), \Dop(\arrv^1, \servv)\bigr)
= \Dop\bigl( \Dop(\barrv^2, \wc\servv^1), \depav^1\bigr).
\end{align*} 
By assumption $\barrv^2, \arrv^1, \servv$ are independent. Hence by Lemma \ref{lm:DR5}(iii)   $\barrv^2, \wc\servv^1, \depav^1$ are independent.  So we take $\zvec=\Dop(\barrv^2, \wc\servv^1)$  which is an  i.i.d.\ Exp$(\alpha_2)$    sequence by Lemma \ref{lm:DR5}(iii).  This proves part  (iv).

 We know that marginally $\arrv^1$ and  $\arrv^2$ are i.i.d.\ sequences.  In queueing language observation (v) becomes obvious. Namely, since $\arrv^2  =    \Dop( \barrv^2, \arrv^1)$, the statement is that past inter-departure   times 
$\{ \arr^2_i\}_{i\le m}$ are independent of future inter-arrival times $\{\arr^1_j\}_{j\ge m+1}$.  Rigorously,  \eqref{m:800} and \eqref{DSR4} show  that variables $\{ \arr^2_i\}_{i\le m}$ are functions of $(\{\barr^2_i\}_{i\le m}\,, \{\arr^1_i\}_{i\le m})$ which are independent of $\{\arr^1_j\}_{j\ge m+1}$. 
\end{proof}

\section{Coupling and monotonicity in last-passage percolation} \label{app:lpp} 

%

In this section $\w=(\w_x)_{x\in\Z^2}$ is a fixed assignment of real weights.   $G_{x,y}$ is  the last-passage value defined by \eqref{v:G}. No probability is involved.

\begin{lemma}\label{lm:G13}    Suppose weights $\w$ and $\wt\w$ satisfy $\w_{o+ie_1}\ge\wt\w_{o+ie_1}$,  $\w_{o+je_2}\le\wt\w_{o+je_2}$, and $\w_x=\wt\w_x$ for $i,j\ge 1$ and $x\in o+\Z_{>0}^2$.   As in \eqref{v:G} define LPP processes 
\[  
	\Gpp_{o,y}=\max_{x_{\brbullet}\,\in\,\Pi_{o,y}}\sum_{k=0}^{\abs{y-o}_1}\w_{x_k}
	\quad\text{and}\quad
	\wt\Gpp_{o,y}=\max_{x_{\brbullet}\,\in\,\Pi_{x,y}}\sum_{k=0}^{\abs{y-x}_1}\wt\w_{x_k}
	\quad\text{ for } \; y \in o+\Z_{\geq 0}^2. 
\] 
Then for all  
$y\in o+\Z_{\geq 0}^2$, the increments over nearest-neighbor edges satisfy 
\[  G_{o,y+e_1}-G_{o,y} \ge \wt G_{o,y+e_1}-\wt G_{o,y}
\quad\text{and}\quad 
G_{o,y+e_2}-G_{o,y} \le \wt G_{o,y+e_2}-\wt G_{o,y}.  
\]

\end{lemma} 

\begin{proof}   
The statements are true by construction  for edges $(y, y+e_i)$ that lie on the  axes $o+\Z_{\ge0}e_i$.   Proceed by induction:  assuming the inequalities hold for the edges $(y,y+e_2)$ and $(y, y+e_1)$ , deduce them for the edges $(y+e_2, y+e_1+e_2)$ and $(y+e_1, y+e_1+e_2)$.  
\end{proof}

\begin{lemma}[Crossing Lemma]\label{lm:crs}
   The inequalities below are valid  whenever the last-passage values are defined. 
	 \begin{align}
	 	G_{o+e_1,\,x+e_2}-G_{o+e_1,\,x} &\leq G_{o,\,x+e_2}-G_{o,\,x}\leq G_{o+e_2,\,x+e_2}-G_{o+e_2,\,x}\label{cr1}\\ 
	 	G_{o+e_2,\,x+e_1}-G_{o+e_2,\,x} &\leq G_{o,\,x+e_1}-G_{o,\,x}\leq G_{o+e_1,\,x+e_1}-G_{o+e_1,\,x}. \label{cr2}	 	
	 \end{align}
\end{lemma}
\begin{proof}
	The proofs of all  parts are similar. We prove the second inequality in \eqref{cr1}, that is, 
	\begin{align}\label{cr in0}
		G_{o,\,x+e_2}-G_{o,\,x}\leq G_{o+e_2,\,x+e_2}-G_{o+e_2,\,x}.
	\end{align}
	 The geodesics $\pi_{o,\,x+e_2}$ and  $\pi_{o+e_2,\,x}$ must cross. Let $u$ be the first point where they meet. Note that
	\begin{align}\label{cr in}
		G_{o,u}+G_{u,\,x}\leq G_{o,\,x}\quad \text{and} \quad G_{o+e_2,u}+G_{u,\,x+e_2}\leq G_{o+e_2,\,x+e_2}.
	\end{align}
	Add the two inequalities in \eqref{cr in} and rearrange to obtain \eqref{cr in0}.
	
	This inequality can be proved also from Lemma \ref{lm:G13}, by writing $G_{o+e_2,\,x+e_2}-G_{o+e_2,\,x}=\wt G_{o,\,x+e_2}-\wt G_{o,\,x}$ with environment $\wt\w_{o+y}=\w_{o+y}$ when $y_2>0$ and $\wt\w_{o+ie_1}=-M$ for  large enough  $M$.    
\end{proof}

Fix base points  $u\le v$ on $\Z^2$.   On the quadrant  $v+\Z_{\ge0}^2$,  put a corner weight $\eta_v=0$ and  define boundary weights 
 \be\label{eta6} \eta_{v+ke_i} =  G_{u,\,v+ke_i}- G_{u,\,v+(k-1)e_i} 
 \qquad\text{for $k\in\Z_{>0}$ and $i\in\{1,2\}$. }\ee
    In the bulk use $\eta_x=\w_x$ for $x\in v+\Z_{>0}^2$.    Denote the LPP process in $v+\Z_{\ge0}^2$ that uses weights $\{\eta_x\}_{x\,\in\, v+\Z_{\ge0}^2}$ by 
 \be\label{wtG8}   G^{[u]}_{v,\,x}=\max_{x_{\brbullet}\,\in\,\Pi_{v,\,x}}  \sum_{i=0}^{\abs{x-v}_1} \eta_{x_i}, \qquad 
 x\in v+\Z_{\ge0}^2. \ee
 The superscript $[u]$ indicates  that   $G^{[u]}$ uses boundary weights determined by the process $G_{u,\bbullet}$  with base point $u$.     Figure \ref{fig-app1}    illustrates the next lemma.  The proof of the lemma is elementary. 

\begin{figure}
\begin{center}
\begin{picture}(200,140)(20,-10)
\put(40,0){\line(1,0){170}} 

 \put(90,40){\line(1,0){120}}\put(210,0){\line(0,1){110}}
\put(40,110){\line(1,0){170}}
\put(40,0){\line(0,1){110}} \put(90,40){\line(0,1){70}}
 
\put(37,-3){\Large$\bullet$} 
\put(30,-2){\small$u$}

\put(86.5,36.5){\Large$\bullet$} 
\put(80,30){\small$v$}

\put(126.5,36.5){\Large$\bullet$} 
\put(134,32){\small$x$}

%

\put(206,106){\Large$\bullet$} 
\put(215,107){\small$y$}


\linethickness{3pt} 
\put(44.5,0){\line(1,0){72}}  \put(115,0){\line(0,1){21}} \put(114,20){\line(1,0){17.5}}
\put(130,20){\line(0,1){16}}  \put(130,44){\line(0,1){16}} 
\put(128.5,61.5){\line(1,0){60}} \put(189,60){\line(0,1){21.5}}
\put(189,80){\line(1,0){22}}  \put(210,78.5){\line(0,1){27}}
 \multiput(94.5,40)(8.5,0){4}{\line(1,0){5.5}}

 \end{picture}
\end{center}  
\caption{ \small Illustration of Lemma \ref{app-lm1}. Path $u$-$x$-$y$ is a geodesic  of  $G_{u,y}$ and path $v$-$x$-$y$ is a geodesic of  $G^{[u]}_{v,y}$. } \label{fig-app1}
\end{figure}

\begin{lemma} \label{app-lm1}  Let $u\le v\le y$ in $\Z^2$.  Then  
$G_{u,y}=G_{u,v}+G^{[u]}_{v,y}$.    The restriction of any geodesic of $G_{u,y}$ to $v+\Z_{\ge0}^2$ is part of a geodesic of  $G^{[u]}_{v,y}$.    The edges with one endpoint in $v+\Z_{>0}^2$ that belong to 
a geodesic of $G^{[u]}_{v, y}$  extend to a geodesic of $G_{u,y}$.  
\end{lemma}  

  Assume now that the weights are such that geodesics are unique. Define the  exit point $Z_{u,\,p}$   as in \eqref{Zdef}.  For $k\ge 1$ let $Z^{[u]}_{u+ke_1, \,p}$ be the exit point of the geodesic of $G^{[u]}_{u+ke_1, \,p}$. The lemma below follows from taking $v=u+ke_1$ in Lemma \ref{app-lm1}. 

\begin{lemma}\label{lm:shift} 
For positive integers $m$,   $Z_{u,\,p}=k+m$  if and only if  $Z^{[u]}_{u+ke_1, \,p}=m$.
\end{lemma}

\section{Random walk bounds}

\begin{lemma}\label{l:rw1}
	Let $\alpha > \beta > 0$, and $S_n = \sum_{k=1}^n Z_k$ be a random walk with step distribution $Z_k \sim {\rm Exp}(\alpha) - {\rm Exp}(\beta)$ {\rm(}difference of two independent exponentials{\rm)}. Then there is an absolute constant $C$ independent of all the parameters such that for $n\in \Z_{>0}$,  
	\beq \label{rw34.8} \mathbb{P} (S_1>0, S_2 >0, \cdots , S_n > 0) \leq \frac{C}{\sqrt n} \left(1- \frac{(\alpha-\beta)^2}{(\alpha+\beta)^2} \right)^n\eeq
	and 
	\beq \label{rw34.9}\mathbb{P} (S_1<0, S_2 <0, \cdots , S_n < 0) \leq \frac{C}{\sqrt n} \left(1- \frac{(\alpha-\beta)^2}{(\alpha+\beta)^2} \right)^n + \frac{\alpha - \beta}{\alpha}.\eeq

\end{lemma}

\begin{proof}
	Define the events 
	\beq\label{AB8}    A^{\alpha,\beta}_n =\{  S_1>0, \dotsc, S_{n}>0\} 
	\quad\text{and}\quad  
	B^{\alpha,\beta}_n =\{  S_1>0, \dotsc, S_{n-1}>0, \, S_n<0\}  \eeq
	for $n\in\Z_{>0}$ and also the decreasing  limit $A^{\alpha,\beta}_\infty =\bigcap_{n\ge 1} A^{\alpha,\beta}_n$. 
	Then 
	\beq\label{AB13}   P(A^{\alpha,\beta}_n) =\sum_{k=n+1}^\infty P(B^{\alpha,\beta}_k)  + P(A^{\alpha,\beta}_\infty). 
	\eeq 
	 Lemma B.3 in  Appendix B of  \cite{fan-sepp-arxiv} calculated 
	\beq \label{eq:B_n}
	P(B^{\alpha,\beta}_n)
	= C_{n-1}\,  \frac{\alpha^{n}\beta^{n-1}}{(\alpha+\beta)^{2n-1}}
	\eeq
	where  $C_n={\frac{1}{n+1}} \binom{2n}{n}$, $n\ge 0$,  are  the Catalan numbers.  	Note that parameters $\alpha$ and $\beta$ are switched around here compared with Lemma B.3   of  \cite{fan-sepp-arxiv}.   From   $\binom{2n}{n}2^{-2n}\sim(\pi n)^{-1/2}$,  we can  fix  a constant  $c_0 $  such that 
	$  C_{k-1}\le {c_0 4^{k-1}}{k^{-3/2}} $.

	The assumption  $\alpha>\beta$ gives  $EZ_k=\alpha^{-1}-\beta^{-1}<0$,  and hence  $\sum_{n\ge 1}P(B^{\alpha,\beta}_n)=1$ and $P(A^{\alpha,\beta}_\infty)=0$.  Thus  \eqref{AB13}  and \eqref{eq:B_n}, together with  $\sum_{k=n+1}^\infty k^{-3/2}\le 2n^{-1/2}$,  give 
	\beq\label{AB15} \begin{aligned}
		P(A^{\alpha,\beta}_n)  
		&= \frac{\alpha}{\alpha+\beta}  \sum_{k=n+1}^\infty C_{k-1}\,\biggl(  \frac{\alpha\beta}{(\beta+\alpha)^2}\biggr)^{k-1} 
		\le  \frac{c_0 \alpha}{\alpha+\beta}  \sum_{k=n+1}^\infty k^{-3/2} \Bigl( 1-\frac{(\alpha-\beta)^2}{(\alpha+\beta)^2}\,\Bigr)^{k-1} \\
		&\le \frac{2c_0 \alpha}{\alpha+\beta} \cdot\frac1{\sqrt n} \Bigl( 1-\frac{(\alpha-\beta)^2}{(\alpha+\beta)^2}\,\Bigr)^{n}. 
	\end{aligned}\eeq

	
%
	
%

	Since  $-S_n$ is obtained from $S_n$ by switching $\alpha$ and $\beta$ around, 
	\begin{align*}
	P\bigl(S_1<0, \dotsc, S_{n}<0\bigr) = P(A^{\beta,\alpha}_n) 
	=\sum_{k=n+1}^\infty P(B^{\beta,\alpha}_k)   + P(A^{\beta,\alpha}_\infty) . 
	\end{align*} 
	Bound the series above as in \eqref{AB15} (with $\alpha$ and $\beta$ interchanged) and add  $P(A^{\beta,\alpha}_\infty)= \frac{\alpha-\beta}\alpha$. 
	This last fact 
	appears  on p.~600 of Resnick \cite{resn} and in Example VI.8(b) on p.~193 of Feller II \cite{fellII}. 
	\end{proof}

\medskip 

\small

\bibliographystyle{plain}

\bibliography{../Timo_old_bib}
 
\end{document}